\documentclass[11pt]{amsart}

\usepackage{amssymb}
\usepackage{amsmath}
\usepackage{mathrsfs}
\usepackage{amsfonts}
\usepackage{color}
\usepackage{amsthm}
\usepackage{graphicx}
\usepackage[colorlinks=true, pdfstartview=FitV, linkcolor=blue,citecolor=blue, urlcolor=blue]{hyperref}
\usepackage{verbatim}

\usepackage[utf8]{inputenc}
\usepackage[T1]{fontenc}
\numberwithin{equation}{section}

\theoremstyle{definition}
\newtheorem{defi}{Definition}[section]
\newtheorem{rem}[defi]{Remark}
\theoremstyle{plain}
\newtheorem{theorem}[defi]{Theorem}

\newtheorem{lem}[defi]{Lemma}

\newtheorem{prop}[defi]{Proposition}

\numberwithin{equation}{section}

\def\R{\mathbb{R}}
\def\Z{\mathbb{Z}}

\def\C{\mathbb{C}}
\def\N{\mathbb{N}}

\renewcommand{\r}{\mathbb{R}}

\newcommand{\eps}{\varepsilon}
\newcommand{\RRe}{\mathrm{Re}\,}
\newcommand{\IIm}{\mathrm{Im}\,}

\allowdisplaybreaks

\begin{document}

\title{Rate of convergence to equilibrium \\ for the heated string}
\author{Tomasz Cie\'{s}lak}
\address{Institute of Mathematics, Polish Academy of Sciences, 00-656 Warsaw, Poland}
\email{cieslak@impan.pl}
\author{Jacek Jendrej}
	\address{Institut de Math\'emtiques de Jussieu, Sorbonne Universit\'e, 75005 Paris, France \& Faculty of Applied Mathematics, AGH University of Science and Technology, 30-059 Krak\'ow, Poland}
	\email{jendrej@imj-prg.fr}
\author{Christian Stinner}
\address{Institute of Mathematics, Polish Academy of Sciences, 00-656 Warsaw, Poland}
\email{stinner@mathematik.tu-darmstadt.de}
\date{\today}
%
\begin{abstract}
In the present manuscript, we calculate the exponential rate of convergence of the heated string system (a mixed-type
hyperbolic-parabolic system of PDEs) towards the equilibrium, independently of the initial data. As a by-product of our analysis, we obtain an enhanced time decay of the solution.

The main tool of our reasoning consists of asymptotic analysis at the Fourier side of the linearized problem in the spirit of Kato. The latter method fits very well to the estimates obtained earlier for the system in \cite{BC25}. The matching between the linear problem and the nonlinear original system requires delicate estimates of the nonlinear terms at the Fourier side as well as careful spectral analysis.
\end{abstract}

\maketitle

: heated string, asymptotic analysis, semigroups, perturbation methods

\subjclass{MSC2020}: 35M13, 35B35, 35B20

\section{Introduction}

This paper is devoted to the asymptotic studies of the mechanics of the heated string. The model we consider to describe the heated string is a 1d thermoelasticity system of PDEs given by the most natural (and simplest) Helmholtz free energy: $\theta \log \theta -\theta +\mu\theta u_x+ \frac{1}{2} u_x^2$.
Under such a choice of the Helmholtz free energy, also picking up the Fick (Fourier) law, one arrives at the hyperbolic-parabolic mixed type system of PDEs describing the phenomenon, the details can be found for instance in \cite{slemrod}.

To be more precise, we study the problem
\begin{equation}\label{eq1}
  \begin{cases}
	  u_{tt} - u_{xx} = \mu \theta_x, & \mbox{in } (0,\infty) \times (0,\pi), \\
		\theta_t - \theta_{xx} = \mu \theta u_{tx}, & \mbox{in } (0,\infty) \times (0,\pi), \\
		u(\cdot, 0) = u(\cdot, \pi) = \theta_x (\cdot, 0) = \theta_x (\cdot, \pi) = 0, & \mbox{in } (0,\infty), \\
		u(0,\cdot) = u_0, \; u_t (0,\cdot) = v_0, \; \theta (0,\cdot) = \theta_0 >0,
	\end{cases}
\end{equation}
where $u(t,x)$ denotes the position of a string fixed at the ends and $\theta (t,x)$ is its temperature. Moreover, $v_0$ is the
initial velocity and $\mu$ a material constant. W.l.o.g. we restrict ourselves to the spatial interval $(0,\pi)$, which seems to
be appropriate for Fourier expansion. We study the problem under the following regularity assumptions on initial data:
\begin{equation}\label{regularity}
 u_0\in H^2(0,\pi)\cap H^1_0(0,\pi),\; v_0\in H^1_0(0,\pi),\; \theta_0\in H^1(0,\pi).
\end{equation}
Under such a regularity of initial data, it is known that the unique solution to \eqref{eq1} exists for any time $t>0$, see \cite{BC23}.
Moreover, in \cite{BC25} it was shown that for any initial data in the regularity class \eqref{regularity} (plus the positivity of initial temperature), the solution converges as time goes to infinity to the flat string with uniformly distributed temperature. In the present paper we study the rate of the latter convergence. As a by-product, we obtain the time decay enhancement for the solution to \eqref{eq1}. Our main theorem reads
\begin{theorem}\label{main}
Let $(u, \theta)$ be the solution to \eqref{eq1} for initial data satisfying the regularity assumptions \eqref{regularity}, with positive initial temperature $\theta_0$. Assume $s\in (\frac{3}{4},1)$. Then, $u$ converges to $0$ in $H^{1+s}(0,\pi)$, $(u_t,\theta)$ converges to $(0, \theta_\infty)$ in $H^s(0,\pi)$, and all the convergences are exponential in time, where the number $\theta_\infty$ is the constant that will be defined in \eqref{theta_in} in Section \ref{preliminaries}.
\end{theorem}

The exponential decay results from the coupling between the wave component and the heat component at the linear level.
We will show that, at high Fourier modes, the coupling induces an effect analogous to including a standard damping term in a wave equation, which yields exponential decay at uniform rate of all Fourier modes (hence no gain of regularity).

%
%
Let us comment on the proof. Our method is an asymptotic analysis type calculation (in the spirit of the Kato book \cite{Kato}) on the Fourier side.
We make use of the convergence of the solution to the equilibrium in $H^s(0,\pi)$ for $s \in (\frac{3}{4},1)$, which is a consequence of the convergence result in \cite{BC25}, the precise statement is given in Section \ref{preliminaries}. The Fourier expansion of the problem \eqref{eq1} is given in Section \ref{Fourier}. Next, in Section \ref{perturbation}, we analyze the spectral properties of the linearization of \eqref{eq1} at the equilibrium. Moreover, the first steps of the argument matching the linear estimates and the original nonlinear problem are taken (it requires the estimates of the exact order of convergence of the eigenvalues and eigenvectors of the linearized system), some linear algebra methods, including the Gershgorin circles to analyze the eigenvalues, are used. Having prepared the framework for the matching between the linear and the original problem, we furnish it with the estimates (the time-dependent variant of weighted $l_2$ estimates of sequences of Fourier expansions) in Section \ref{estimates}. Finally, our asymptotic construction closes up in Section \ref{contraction} in a proper fixed point Banach contraction argument. In Section \ref{conlusion}, we use the previous construction to prove the main theorem.

Finally, let us compare our result to the already existing in the literature. It was Slemrod \cite{slemrod}, who first considered the nonlinear 1d thermoelasticity model. However, the system of PDEs studied in \cite{slemrod} does not cover \eqref{eq1}. Though, seemingly more general, due to the considered assumptions, the system in \cite{slemrod} does not allow the nonlinear term in the heat equation to depend on $\theta$ in a way, which would allow the positivity of $\theta$. Moreover, under the assumptions in \cite{slemrod}, the second principle of thermodynamics does not hold (unlike in our system, see \cite{BC25}, where the quantitative version of the second principle of thermodynamics is used in the proof of convergence of the solution towards equilibrium). Still, under his restrictive assumptions, Slemrod shows local in time existence of solutions. Next, he also shows, that when starting close to the equilibrium, solutions exist for any time and converge to the equilibrium.

The articles \cite{jiang} and \cite{racke_shibata} address the system of PDEs very similar to the one studied in \cite{slemrod}. Again, the structural assumptions exclude applicability to \eqref{eq1} (or a similar system of PDEs satisfying positivity of temperature and second principle of thermodynamics). But, in those restrictive models, additional information on the rate of convergence to equilibria is obtained, still, only for initial data starting close to equilibrium. In \cite{racke_shibata}, the polynomial rate of convergence for the Dirichlet problem in both, temperature and displacement, is obtained. In \cite{jiang}, the exponential rate holds, this time for the Neumann problem in temperature. Again, the initial data need to be close to equilibrium in quite a restrictive regularity class. And, as we already mentioned, the system studied in \cite{jiang} does not overlap with our system \eqref{eq1}. It is also important to mention the work \cite{Munoz} of Mu\~noz Rivera, where the linearization of our system, however, under Dirichlet data, is studied. The exponential rate then follows (but the problem is then much simpler and the smart use of an energy method gives the desired rate of convergence). To summarize, we faced a completely new and quite challenging problem. We had to employ a refined asymptotic analysis technique, including the perturbation of linear unbounded operators and match it with the non-trivial estimate obtained in \cite{BC25}.

\section{Preliminaries}\label{preliminaries}
In this section we recall some estimates for the solutions of \eqref{eq1} and refine them a bit, so that they are applicable in our studies.
First, we recall the fact that the energy associated to \eqref{eq1}, is conserved along the trajectory of a solution to \eqref{eq1}:
\begin{equation}\label{energy}
\frac{d}{dt}E(t)=0, \;\mbox{where}\; E(t)=\frac{1}{2}\int_0^\pi u_t^2 dx+\frac{1}{2}\int_0^\pi u_x^2 dx +\int_0^\pi \theta dx.
\end{equation}
Next, we mention that in \cite[Theorem 1.1]{BC25} it is proven that $u(t, \cdot)\rightarrow 0, u_t(t, \cdot)$ $\rightarrow 0$ in $H^1(0,\pi)$ and $L^2(0,\pi)$, respectively. Moreover,
\begin{equation}\label{theta_in}
\theta(t)\rightarrow \theta_\infty:=\frac{1}{\pi}E(0)\;\mbox{in}\; L^2(0,\pi) \mbox{ as } t \to \infty.
\end{equation}
In addition, \cite[Theorem 4.1 and (21)]{BC25} shows that there exists $C>0$ such that for any $t>0$
\begin{equation}\label{oszacowanie}
\|u(t, \cdot)\|_{H^2(0,\pi)}+\|u_t(t, \cdot)\|_{H^1(0,\pi)}+\|\theta(t, \cdot)\|_{H^1(0,\pi)}\leq C.
\end{equation}
Interpolating between \eqref{oszacowanie} and convergence in the above mentioned \cite[Theorem 1.1]{BC25}, we obtain the following proposition.
\begin{prop}\label{zbieznosc}
Assume that initial data $u_0, v_0, \theta_0$ satisfy \eqref{regularity}. Moreover, assume $\theta_0>0$. Let $(u, \theta)$ be the unique solution to \eqref{eq1} (such a solution was obtained in \cite{BC23} for any such initial data) emanating from $(u_0, v_0, \theta_0)$. Then, as time $t\rightarrow \infty$, $(u_x, u_t, \theta)$ converges towards $(0, 0, \theta_\infty)$ in $H^s(0,\pi)$, for any $0\leq s<1$.
\end{prop}
The above enhancement of the convergence result in \cite[Theorem 1.1]{BC25} is crucial for our estimates, as will be seen in our proof, where we require the convergence to hold in $H^s$, $s \in (\frac{3}{4},1)$.

Notice that as a straightforward consequence, we see that
\begin{equation}\label{zbtheta}
\int_0^\pi \theta(t,x)dx\rightarrow \pi\theta_\infty \mbox{ as } t \to \infty,
\end{equation}
where $\theta_\infty$ is defined in \eqref{theta_in}.

At the end of this section, let us comment on one of the consequences of our main Theorem \ref{main}.
\begin{rem}
The whole energy $E(t)$, defined in \eqref{energy}, is converted into heat at an exponential rate along the trajectories of \eqref{eq1}.
\end{rem}

\section{Fourier expansion}\label{Fourier}
In view of the boundary conditions in \eqref{eq1} we use the Fourier expansion
\begin{equation}\label{f1}
  u(t,x) = \sum\limits_{n=1}^\infty \hat{u}_n (t) \sin(nx), \qquad
	\theta (t,x) = \sum\limits_{n=0}^\infty \hat{\theta}_n (t) \cos(nx),
\end{equation}
for $(t,x) \in [0,\infty) \times [0,\pi]$. Inserting this into \eqref{eq1} and using the Cauchy product of series, we get for
$(t,x) \in (0,\infty) \times (0,\pi)$
\begin{align*}
  0 &= u_{tt} - u_{xx} - \mu \theta_x
	= \sum\limits_{n=1}^\infty \sin(nx) \left( \hat{u}_n^{\prime\prime} (t) + n^2 \hat{u}_n (t) + \mu n \hat{\theta}_n (t)\right), \\
	0 &= \theta_t - \theta_{xx} - \mu \theta u_{tx} \\
	&= \sum\limits_{n=0}^\infty \cos(nx) \left( \hat{\theta}_n^\prime (t) + n^2 \hat{\theta}_n (t) \right) \\
	& \hspace*{+5mm} -\mu \left( \sum\limits_{n=0}^\infty \hat{\theta}_n (t) \cos(nx) \right)
	\left( \sum\limits_{n=0}^\infty n \hat{u}_n^\prime (t) \cos(nx)\right) \\
	&= \sum\limits_{n=0}^\infty \cos(nx) \left( \hat{\theta}_n^\prime (t) + n^2 \hat{\theta}_n (t) \right) \\
	& \hspace*{+5mm} -\mu \sum\limits_{n=0}^\infty \sum\limits_{k=0}^n \hat{\theta}_{n-k}(t) \cos((n-k)x) k \hat{u}_k^\prime (t)
	\cos(kx) .
\end{align*}
In view of
$$\cos(y \pm z) = \cos(y)\cos(z) \mp \sin(y)\sin(z),$$
we have
$$\cos(y)\cos(z) = \frac{1}{2} \left( \cos(y+z) + \cos(y-z)\right), \quad y,z \in \R .$$
Inserting this into the previous identity, rearranging the terms, and using that $\cos$ is even, we obtain
\begin{align*}
  0 &= \sum\limits_{n=0}^\infty \cos(nx) \left( \hat{\theta}_n^\prime (t) + n^2 \hat{\theta}_n (t) \right) \\
  & \hspace*{+5mm} -\frac{\mu}{2} \sum\limits_{n=1}^\infty \sum\limits_{k=1}^n \hat{\theta}_{n-k}(t) k \hat{u}_k^\prime (t)
	\big( \cos(nx) + \cos((n-2k)x) \big) \\
  &= \sum\limits_{n=0}^\infty \cos(nx) \left( \hat{\theta}_n^\prime (t) + n^2 \hat{\theta}_n (t)
  -\frac{\mu}{2} \sum\limits_{k=1}^n \hat{\theta}_{n-k}(t) k \hat{u}_k^\prime (t)  \right) \\
	& \hspace*{+5mm} -\frac{\mu}{2} \bigg( \cos(0x) \sum\limits_{l=1}^\infty \hat{\theta}_l (t) l \hat{u}_l^\prime (t) \\
	& \hspace*{+5mm} + \sum\limits_{n=1}^\infty \cos(nx) \left[ \sum\limits_{l=1}^\infty \hat{\theta}_{l+n} (t)
	l \hat{u}_l^\prime (t) + \sum\limits_{l=0}^\infty \hat{\theta}_l (t) (l+n) \hat{u}_{l+n}^\prime (t) \right] \bigg) .	
\end{align*}
Altogether, as $\cos(nx)$, $n \in \N_0$, and $\sin(nx)$, $n \in \N$, are linearly independent, we obtain the system of ODEs
\begin{equation}\label{f2}
  \begin{cases}
	  \hat{u}_n^{\prime\prime} (t) =  -n^2 \hat{u}_n (t) - \mu n \hat{\theta}_n (t), \quad t \in (0,\infty), \; n \in \N, \\
		\hat{\theta}_0^\prime (t) = \frac{\mu}{2} \sum\limits_{l=1}^\infty \hat{\theta}_l (t) l \hat{u}_l^\prime (t),
		\quad t \in (0,\infty), \\
		\hat{\theta}_n^\prime (t) = - n^2 \hat{\theta}_n (t)
    +\frac{\mu}{2} \sum\limits_{k=1}^n \hat{\theta}_{n-k}(t) k \hat{u}_k^\prime (t)
		+ \frac{\mu}{2} \sum\limits_{l=1}^\infty \hat{\theta}_{l+n} (t) l \hat{u}_l^\prime (t) \\
		\hspace*{+15mm} + \frac{\mu}{2} \sum\limits_{l=0}^\infty \hat{\theta}_l (t) (l+n) \hat{u}_{l+n}^\prime (t),
		\quad t \in (0,\infty), \; n \in \N .
	\end{cases}
\end{equation}
In particular, we note that when time goes to $\infty$,
\begin{equation}\label{zbwszthe}
\pi \hat{\theta}_0(t) \rightarrow E(0), \quad
\hat{\theta}_n(t)\rightarrow 0, \; n \in \N.
\end{equation}
Indeed, this is a consequence of \eqref{zbtheta} and \eqref{theta_in}.

At the end of this section let us state an interesting remark.
\begin{rem}
The conservation of energy \eqref{energy} could be easily obtained from \eqref{f2}. Indeed, multiplying each equation for $\hat{u}_n$ by $\hat{u}_n^{\prime}$, next summing them from $n=1$ to infinity and using the equation on $\hat{\theta}_0^{\prime}$, we arrive at the energy conservation.
Here, one notices that the derivation of the energy conservation does not require using the evolution of $\hat{\theta}_n$ for $n\geq 1$.
\end{rem}

\section{Perturbation for linear system}\label{perturbation}
We write $\theta = a + \psi$ for some constant $a \in (0,\infty)$ and linearize \eqref{eq1} to obtain
\begin{equation}\label{linpde}
  \begin{cases}
	  u_{tt} - u_{xx} = \mu \psi_x, & \mbox{in } (0,\infty) \times (0,\pi), \\
		\psi_t - \psi_{xx} = \mu a u_{tx}, & \mbox{in } (0,\infty) \times (0,\pi).
	\end{cases}	
\end{equation}
Using Fourier expansion, in view of \eqref{f1} and $\theta = a + \psi$ we have
\begin{equation}\label{linf1}
  \psi (t,x) = \sum\limits_{n=0}^\infty \hat{\psi}_n (t) \cos(nx)
	= \hat{\psi}_0(t) + \sum\limits_{n=1}^\infty \hat{\theta}_n (t) \cos(nx), \; \hat{\theta}_0 = a + \hat{\psi}_0 .
\end{equation}
Inserting this and \eqref{f1} into \eqref{linpde}, we have
$$0 = \psi_t - \psi_{xx} - \mu a u_{tx}
  = \sum\limits_{n=0}^\infty \cos(nx) \left( \hat{\psi}_n^\prime (t) + n^2 \hat{\psi}_n (t) - a \mu n \hat{u}_n^\prime (t)
	\right) .
$$
As $\hat{\psi}_n = \hat{\theta}_n$ for $n \in \N$, for the first PDE of \eqref{linpde} we get again the first ODE in \eqref{f2}.
Hence, for this linearization we have the system of ODEs
\begin{equation}\label{linf2}
  \begin{cases}
	  \hat{u}_n^{\prime\prime} (t) =  -n^2 \hat{u}_n (t) - \mu n \hat{\theta}_n (t), \quad t \in (0,\infty), \; n \in \N, \\
		\hat{\theta}_n^\prime (t) = - n^2 \hat{\theta}_n (t) + a \mu n \hat{u}_n^\prime (t), \quad t \in (0,\infty), \; n \in \N .
	\end{cases}
\end{equation}
With
\begin{equation}\label{linf3}
  u_t (t,x) = \sum\limits_{n=1}^\infty \hat{u}_n^\prime (t) \sin(nx) = \sum\limits_{n=1}^\infty \hat{v}_n (t) \sin(nx)
\end{equation}
we get the first order linear system
\begin{equation}\label{linsys}
  \frac{d}{dt} \begin{pmatrix} n \hat{u}_n \\ \hat{v}_n \\ \hat{\theta}_n \end{pmatrix}
	= A_{n,a} \begin{pmatrix} n \hat{u}_n \\ \hat{v}_n \\ \hat{\theta}_n \end{pmatrix}, \quad t \in (0,\infty), \; n \in \N,
\end{equation}
with
\begin{equation}\label{linsysa}	
  A_{n,a} := \begin{pmatrix} 0 & n & 0 \\ -n & 0 & -\mu n \\ 0 & a \mu n & -n^2 \end{pmatrix} .
\end{equation}
With this notation, the nonlinear system \eqref{f2} corresponds to the first order system
\begin{equation}\label{nonlinsys}
  \frac{d}{dt} \begin{pmatrix} n \hat{u}_n \\ \hat{v}_n \\ \hat{\theta}_n \end{pmatrix}
	= A_{n,a} \begin{pmatrix} n \hat{u}_n \\ \hat{v}_n \\ \hat{\theta}_n \end{pmatrix} + g_n, \quad t \in (0,\infty), \; n \in \N,
\end{equation}
with
\begin{equation}\label{nonlinsysf}	
  \begin{split}
	 & g_n := \begin{pmatrix} 0 \\ 0 \\ g_{3,n} \end{pmatrix}, \qquad\mbox{where} \\
	 & g_{3,n} :=	\frac{\mu}{2} \sum\limits_{k=1}^{n-1}
	\hat{\theta}_{n-k} k \hat{v}_k + \frac{\mu}{2} \sum\limits_{l=1}^\infty \hat{\theta}_{l+n} l \hat{v}_l \\
	& \hspace*{+10mm} + \frac{\mu}{2} \sum\limits_{l=1}^\infty \hat{\theta}_l (l+n) \hat{v}_{l+n}
	+ \mu (\hat{\theta}_0 -a) n \hat{v}_n .
	\end{split}
\end{equation}
Next we study the linear system \eqref{linsys}.
\begin{lem}\label{lem3.1}
  Let $n \in \N$, $a>0$, and $A_{n,a}$ defined in \eqref{linsysa}. Then any eigenvalue $\lambda_n$ of $A_{n,a}$ satisfies
	$\RRe(\lambda_n) <0$.
\end{lem}
\begin{proof}
  Allowing for later use also complex vectors, for $y:=(n \hat{u}_n, \hat{v}_n, \hat{\theta}_n) \in \C^3$, we define
 	$$E_n (y) := |n \hat{u}_n|^2 + |\hat{v}_n|^2 + \frac{1}{a} |\hat{\theta}_n|^2 .$$
	Restricting to the case $y \in \R^3$, we have $E_n \ge 0$ and $E_n (y)=0$ only for $y =0$ as well as
	\begin{align*}
	  \dot{E}_n (y) &= \nabla E_n(y) \cdot A_{n,a}y \\
	  &= 2n \hat{u}_n n \hat{v_n} + 2 \hat{v_n} (-n^2 \hat{u}_n -\mu n \hat{\theta}_n) + \frac{2}{a} \hat{\theta_n}(a\mu n \hat{v}_n
	  -n^2 \hat{\theta}_n) \\
		&= - \frac{2n^2}{a} \hat{\theta}_n^2 \le 0 .
	\end{align*}
	Hence, $E_n$ is a Liapunov function for \eqref{linsys} and thus $y=0$ is stable for \eqref{linsys} and any eigenvalue
	$\lambda_n$ of $A_{n,a}$ satisfies $\RRe(\lambda_n) \le 0$ by \cite[Theorems 2 and 3 in Section~2.9]{perko}. 	
	
	Assume for contradiction that there is an eigenvalue $\lambda_n$ of $A_{n,a}$ such that $\RRe(\lambda_n) =0$. Then there is a
	corresponding eigenvector $y \in \C^3$ and $z(t) := e^{\lambda_n t} y$, $t \ge 0$, is a solution to \eqref{linsys}. The
	definition of $E_n$ and $|e^{\lambda_n t}| = 1$ imply
	$$E_n(z(t)) = E_n(y) = E_n(\RRe(y)) + E_n(\IIm(y)), \qquad t \ge 0,$$
	and hence
	\begin{align*}
	  \dot{E}_n(y) &= \dot{E}_n (\RRe(y)) + \dot{E}_n(\IIm(y))
		= - \frac{2n^2}{a} \left( |\RRe(\hat{\theta}_n)|^2 + |\IIm(\hat{\theta}_n)|^2 \right) \\
		&= - \frac{2n^2}{a} |\hat{\theta}_n|^2.
	\end{align*}	
	As $E_n(z(t))$ is constant, we have
	$$0 = \frac{d}{dt} E_n(z(t)) \Big|_{t=0} = \dot{E}_n(y) = - \frac{2n^2}{a} |\hat{\theta}_n|^2,$$
	and hence $\hat{\theta}_n =0$. In view of $A_{n,a}y = \lambda_n y$, the third component now implies $\hat{v}_n =0$ and
	then we get $\hat{u}_n=0$ from the second component. But then $y=0$, a contradiction, as $y$ is an eigenvector. Hence,
	the claim is proved.
\end{proof}
\begin{rem}
An alternative proof could be obtained by applying the Routh-Hurwitz criterion.
\end{rem}

Our aim is to study projections of the solutions to \eqref{nonlinsys} with respect to appropriate directions and start by
studying the eigenvalues and eigenvectors of
\begin{equation}\label{linsysat}	
  A_{n,a}^\ast = \begin{pmatrix} 0 & -n & 0 \\ n & 0 & a\mu n \\ 0 & -\mu n & -n^2 \end{pmatrix}
\end{equation}
for large $n$. In what follows, we shall use the enhanced version of the Gershgorin theorem (see e.g. \cite[Theorem~4]{LZ19}), stating that if the Gershgorin disks (that are defined below) divide into the disjoint unions of disks, then the eigenvalues (calculated with their multiplicities) must lie in the separated unions of Gershgorin disks. In our case, the matrix is $3\times 3$, and as we shall see, one of the Gershgorin disks is separated from the union of other two, for large enough $n$. Hence, one of eigenvalues will lie in the single separated disk, while the other two will lie in the separated double.

The eigenvalues $\lambda_{j,n}$, $j=1,2,3$, of $A_{n,a}^\ast$, $n \in \N$, have the following properties.
\begin{lem}\label{lem3.3}
  Let $n \in \N$, $a>0$, and $A_{n,a}^\ast$ defined in \eqref{linsysat}. Then there are $n_0 \in \N$ and $c>0$ such that, for any
	$n \ge n_0$, $A_{n,a}^\ast$ has three simple eigenvalues $\lambda_{j,n}$, $j=1,2,3$, and we have
	$|\lambda_{1,n} - (-n^2 + a \mu^2)| \le \frac{c}{n^2}$, $|\lambda_{2,n} -(-ni - \frac{a \mu^2}{2})| \le \frac{c}{n}$, and
  $|\lambda_{3,n} -( ni - \frac{a \mu^2}{2})| \le \frac{c}{n}$.
\end{lem}
\begin{proof}
  For $n \in \N$ let $A_{n,a}^\ast = (a_{kj})$ and
	$$D_{k,n} := \Big\{z \in \C \: : \: |z-a_{kk}| \le \sum\limits_{j \neq k} |a_{kj}| \Big\}, \quad k = 1,2,3,$$
	the Gershgorin disks. Then we have $D_{3,n} = \overline{B(-n^2, \mu n)}$, $D_{1,n} =\overline{B(0,n)}$, and
	$D_{2,n} = \overline{B(0,n(1+a\mu))}$. Hence,
	there is $N_1 \in \N$ such that $D_{3,n} \cap (D_{1,n} \cup D_{2,n}) = \emptyset$ for all $n \ge N_1$. Hence, by Gershgorin's
	theorem for any $n \ge N_1$ exactly one eigenvalue $\lambda_{1,n}$ of $A_{n,a}^\ast$ is
	contained in $D_{3,n}$, while the other two
	eigenvalues $\lambda_{2,n}$ and $\lambda_{3,n}$ (complex conjugated if they are not in $\R$, as the matrix is of real
	entries) are contained in $D_{1,n} \cup D_{2,n}$.
	
	Assume for contradiction that there is $n \ge N_1$ such that $\lambda_{2,n} \in \R$. As $\lambda_{2,n} \in D_{1,n} \cup
	D_{2,n}$, we have $|\lambda_{2,n}| \le n(1+a\mu)$. The
	characteristic polynomial of $A_{n,a}^\ast$ is
	\begin{align*}
	  p_n(\lambda) &= \det(A_{n,a}^\ast - \lambda I)
		= \det \begin{pmatrix} -\lambda & -n & 0 \\ n & -\lambda & a\mu n \\ 0 & -\mu n & -n^2-\lambda \end{pmatrix} \\
		&= -\lambda (-\lambda(-n^2-\lambda) + a\mu^2n^2) -n (-n(-n^2-\lambda)-0) \\
		&= -n^2 \lambda^2 - \lambda^3 -a\mu^2n^2 \lambda -n^4 -n^2 \lambda \\
		&= -(\lambda^3 + n^2 \lambda^2 + n^2(a\mu^2+1)\lambda + n^4).
	\end{align*}
	In view of $\lambda_{2,n} \in \R$ and $|\lambda_{2,n}| \le n(1+a\mu)$, there is $N_2 \ge N_1$ such that for any $n \ge N_2$
	we have
	$$p_n(\lambda_{2,n}) \le -n^4 + n^3\left[(1+a\mu)^3 +(a\mu^2+1)(1+a\mu) \right] \le -\frac{n^4}{2} <0.$$
	Hence, for $n \ge N_2$ this is a contradiction and in that case we have $\lambda_{2,n} \in \C \setminus \R$. In particular,
	since $A_{n,a}^\ast$ is a real matrix, we have $\lambda_{3,n} = \overline{\lambda_{2,n}} \neq \lambda_{2,n}$ and $A_{n,a}^\ast$
	has three simple eigenvalues in that case.
	
	Hence, for $n \ge N_2$ we have $\lambda_{3,n} = -x_n+iy_n$ and $\lambda_{2,n} = -x_n-iy_n$ with $x_n,y_n \in \R$. As the above
	result implies $y_n \neq 0$, we choose w.l.o.g. $y_n>0$. As $A_{n,a}^\ast = (A_{n,a})^T$ has the same eigenvalues as $A_{n,a}$,
	Lemma~\ref{lem3.1} implies $x_n>0$. In view of $\lambda_{3,n} \in D_{1,n} \cup D_{2,n}$, we further have
	$|x_n|, |y_n| \le n(1+a\mu)$.
	
	Next, we get
	$$-n^2 = \mathrm{trace}(A_{n,a}^\ast) = \lambda_{1,n} + \lambda_{2,n} + \lambda_{3,n} = \lambda_{1,n} - 2x_n,$$
	and therefore
	\begin{equation}\label{eq3.1.1}
	  \lambda_{1,n} = -n^2 + 2x_n .
	\end{equation}	
	As $p_n(\lambda_{1,n})=0$, we have
	\begin{align*}
	  0 &= (-n^2 + 2x_n)^3 + n^2 (-n^2 + 2x_n)^2 + n^2(a\mu^2+1)(-n^2 + 2x_n) + n^4 \\
		&= -n^6 + 3n^4 \cdot 2x_n - 3n^2 \cdot 4x_n^2 + 8x_n^3 +n^6 -4x_n n^4 + 4x_n^2 n^2 \\
		&\hspace*{+5mm} - n^4(a\mu^2 +1) + 2x_n n^2(a\mu^2 +1) + n^4 \\
		&= -a\mu^2 n^4 + 2x_n n^4 + 2x_n n^2(a\mu^2 +1) - 8x_n^2 n^2 + 8x_n^3 .
	\end{align*}
	If $(x_n)_{n \ge N_2}$ is not bounded, then the term $2x_n n^4$ cannot be compensated as $n \to \infty$, since all other terms
	are in $O(n^4)$ in view of $|x_n| \le n(1+a\mu)$. Hence, $(x_n)_{n \ge N_2}$ must be bounded and we get
	$$0= -a\mu^2 n^4 + 2x_n n^4 + O(n^2).$$
	This yields $x_n = \frac{a\mu^2}{2} + b_n$ and hence $0= 2b_n n^4 + O(n^2)$. In view of \eqref{eq3.1.1}, we conclude that there
	is $N_3 \ge N_2$ such that
	\begin{equation}\label{eq3.1.2}
	  x_n = \frac{a\mu^2}{2} + O \left(\frac{1}{n^2} \right) \quad\mbox{and}\quad
		\lambda_{1,n} = -n^2 + a\mu^2 + O \left(\frac{1}{n^2} \right), \quad n \ge N_3 .
	\end{equation}
	Next, as $p_n(\lambda_{3,n})=0$, we have
	\begin{align*}
	  0 &= (-x_n+iy_n)^3 + n^2 (-x_n+iy_n)^2 + n^2(a\mu^2+1)(-x_n+iy_n) + n^4 \\
		&= -x_n^3 + 3x_n^2 y_n i +3x_n y_n^2 -i y_n^3 +n^2 (x_n^2 -y_n^2) - 2n^2 x_n y_n i \\
		&\hspace*{+5mm} -n^2 (a\mu^2+1) x_n +iy_n n^2(a\mu^2 +1) +n^4 \\
		&= iy_n \left(3x_n^2 - y_n^2 -2n^2 x_n +n^2(a\mu^2 +1) \right) -x_n^3 + 3x_ny_n^2 \\
		&\hspace*{+5mm} + n^2 (x_n^2-y_n^2) -n^2 (a\mu^2 +1)x_n + n^4 .
	\end{align*}	
	In view of $x_n, y_n \in \R$, $y_n \neq 0$, and \eqref{eq3.1.2}, the imaginary part implies
	\begin{align*}
	  0&= 3x_n^2 - y_n^2 -2n^2 x_n +n^2(a\mu^2 +1) \\
	  &= 3x_n^2 - y_n^2 -a\mu^2 n^2 -2n^2 b_n + n^2(a\mu^2 +1)
	  = n^2 -y_n^2 +O(1).
	\end{align*}	
	Since $y_n >0$, we get $y_n = n + d_n$ with $d_n \in o(n)$ and hence
	$$0= n^2 -(n+d_n)^2 + O(1) = -2n d_n - d_n^2 + O(1).$$
	This yields $d_n \in O(\frac{1}{n})$ and finally $y_n = n + O(\frac{1}{n})$ for $n \ge n_0$ with some $n_0 \ge N_3$.
	In conjunction with $\lambda_{3,n} = -x_n+iy_n$, $\lambda_{2,n} = -x_n-iy_n$, and \eqref{eq3.1.2}, the claim is proved.
\end{proof}
Next, we shall see how do the eigenvectors, corresponding to the eigenvalues of $A_{n,a}^\ast$, look like asymptotically.
\begin{lem}\label{lem3.2}
  Let $n_0 \in \N$ be from Lemma~\ref{lem3.3}, $n \in \N$ with $n \ge n_0$, $a>0$, and $A_{n,a}^\ast$ defined in
	\eqref{linsysat}. Then we have $A_{n,a}^\ast V_{j,n} = \lambda_{j,n}V_{j,n}$
	for $j = 1,2,3$, where
	\begin{align*}
	  & \lambda_{1,n} = -n^2 + a \mu^2 + O \left(\frac{1}{n^2} \right), \qquad
		V_{1,n} = \begin{pmatrix} -\frac{a \mu}{n^2} + O \left(\frac{1}{n^3} \right) \\[1mm] -\frac{a\mu}{n} + O \left(\frac{1}{n^3}
		\right) \\[1mm] 1- \frac{a\mu^2}{n^2} + O \left(\frac{1}{n^3} \right) \end{pmatrix}, \\
		& \lambda_{2,n} = -ni - \frac{a \mu^2}{2} + O \left(\frac{1}{n} \right), \qquad
		V_{2,n} = \begin{pmatrix} 1 \\[1mm] i + \frac{a\mu^2}{2n} + O \left(\frac{1}{n^2}
		\right) \\[1mm] -\frac{\mu i}{n} + \frac{\mu}{n^2} - \frac{a\mu^3}{2n^2} + O \left(\frac{1}{n^3} \right) \end{pmatrix}, \\
		& \lambda_{3,n} = ni - \frac{a \mu^2}{2} + O \left(\frac{1}{n} \right), \qquad
		V_{3,n} = \begin{pmatrix} 1 \\[1mm] -i + \frac{a\mu^2}{2n} + O \left(\frac{1}{n^2}
		\right) \\[1mm] \frac{\mu i}{n} + \frac{\mu}{n^2} - \frac{a\mu^3}{2n^2} + O \left(\frac{1}{n^3} \right) \end{pmatrix}.
	\end{align*}
\end{lem}
\begin{proof}
  For $j=1$ we know by Lemma \ref{lem3.3} that $\lambda_{1,n} =-n^2 + a \mu^2 + O \left(\frac{1}{n^2}\right)$. Hence,
	we use the ansatz $\lambda_{1,n} = -n^2 + \lambda$ with $\lambda = a \mu^2 + O \left(\frac{1}{n^2}\right)$ and
	$V_{1,n} = (\eps_1, \eps_2, 1+ \eps_3)^T$ and aim to solve the
	equation
	$$(-n^2+\lambda) \begin{pmatrix} \eps_1 \\ \eps_2 \\ 1+ \eps_3 \end{pmatrix}
	= A_{n,a}^\ast \begin{pmatrix} \eps_1 \\ \eps_2 \\ 1+ \eps_3 \end{pmatrix}
	= \begin{pmatrix} -n \eps_2 \\ n \eps_1 + a \mu n (1+ \eps_3) \\ -\mu n \eps_2 -n^2 -n^2 \eps_3 \end{pmatrix}.$$
	The third component and the value of $\lambda$ yield $-\mu n \eps_2= \lambda (1+ \eps_3) = a\mu^2 + a\mu^2 \eps_3
	+ O(\frac{1}{n^2})$.
	Assuming $\eps_3 \in O(\frac{1}{n})$ we get $\eps_2 = -\frac{a\mu}{n}
	+ O(\frac{1}{n^2})$. Then the first component yields $(n - \frac{\lambda}{n}) \eps_1 = \eps_2 = -\frac{a\mu}{n}
	+ O(\frac{1}{n^2})$
	and hence $\eps_1 = -\frac{a\mu}{n^2} + O(\frac{1}{n^3})$. Writing $\eps_2 = -\frac{a\mu}{n} + \frac{b_1}{n^2} + \frac{b_2}{n^3}
	+ O(\frac{1}{n^4})$
	and $\eps_3 = \frac{d_1}{n} + \frac{d_2}{n^2} + O(\frac{1}{n^3})$,
	we get from the second component
	$a\mu n -b_1 - \frac{b_2}{n} - \frac{a^2\mu^3}{n} + O(\frac{1}{n^2}) = -\frac{a\mu}{n} + a \mu n + a \mu d_1 + \frac{a\mu d_2}{n}
	+ O(\frac{1}{n^2})$. We choose e.g. $b_1 = d_1 =0$, $b_2 = a \mu$, and $d_2 = -a\mu^2$. Then $V_{1,n}$ has the claimed form and
	the equation $A_{n,a}^\ast V_{1,n} = \lambda_{1,n}V_{1,n}$,
	which is equivalent to
	$$\begin{pmatrix} a\mu + O(\frac{1}{n}) \\[1mm] a \mu n + O(\frac{1}{n}) \\[1mm] - n^2 + 2 a \mu^2 + O(\frac{1}{n}) \end{pmatrix}
	= \begin{pmatrix} a\mu + O(\frac{1}{n}) \\[1mm] a \mu n + O(\frac{1}{n}) \\[1mm] - n^2 + 2 a \mu^2 + O(\frac{1}{n})
	\end{pmatrix},$$
	is satisfied with an error of $O(\frac{1}{n})$.
	
	For $j=2$ we know from Lemma \ref{lem3.3} that $\lambda_{2,n} = -ni -\frac{a\mu^2}{2}
	+ O \left(\frac{1}{n}\right)$. Hence, we use the ansatz $\lambda_{2,n} = -ni + \lambda$ with $\lambda = -\frac{a\mu^2}{2}
	+ O \left(\frac{1}{n}\right)$ and $V_{2,n} = (1+\eps_1, i+\eps_2, \eps_3)^T$ and aim to solve the
	equation
	$$(-ni +\lambda) \begin{pmatrix} 1+\eps_1 \\ i+\eps_2 \\ \eps_3 \end{pmatrix}
	= A_{n,a}^\ast \begin{pmatrix} 1+\eps_1 \\ i+\eps_2 \\ \eps_3 \end{pmatrix}
	= \begin{pmatrix} -ni -n\eps_2 \\ n+ n \eps_1 + a \mu n \eps_3 \\ -\mu n i -\mu n \eps_2 -n^2 \eps_3 \end{pmatrix}.$$
	We choose e.g. $\eps_1 =0$. Then the three components yield $\eps_2 = -\frac{\lambda}{n}$, $a\mu n \eps_3 = 2\lambda i
	- \frac{\lambda^2}{n}$, and $(1 - \frac{i}{n} + \frac{\lambda}{n^2}) \eps_3 = -\frac{\mu i}{n} + \frac{\mu \lambda}{n^2}$.
	Inserting the value of $\lambda$, we get $\eps_3 = \frac{2\lambda i}{a\mu n} + O(\frac{1}{n^2}) = -\frac{\mu i}{n}
	+ O(\frac{1}{n^2})$. This implies
	$\eps_3 = -\frac{\mu i}{n} + \frac{\mu \lambda}{n^2} +(\frac{i}{n} - \frac{\lambda}{n^2}) \eps_3 = -\frac{\mu i}{n}
	- \frac{a\mu^3}{2n^2} + \frac{\mu}{n^2} + O(\frac{1}{n^3})$. Then $V_{2,n}$ has the claimed form and
	$A_{n,a}^\ast V_{2,n} = \lambda_{2,n}V_{2,n}$,
	which is equivalent to
	$$\begin{pmatrix} -n i - \frac{a\mu^2}{2} + O(\frac{1}{n}) \\[1mm] n - a \mu^2 i + O(\frac{1}{n}) \\[1mm]
	- \mu + O(\frac{1}{n}) \end{pmatrix}
	= \begin{pmatrix} -n i - \frac{a\mu^2}{2} + O(\frac{1}{n}) \\[1mm] n - a \mu^2 i + O(\frac{1}{n}) \\[1mm]
	- \mu + O(\frac{1}{n})
	\end{pmatrix},$$
	is satisfied with an error of $O(\frac{1}{n})$.
	
  As $A_{n,a}^\ast$ is real, we get $\lambda_{3,n} = \overline{\lambda_{2,n}}$ and $V_{3,n} = \overline{V_{2,n}}$ by complex
	conjugation.
\end{proof}
Taking only the explicitly given terms in $\lambda_{j,n}$ and $V_{j,n}$, we define
\begin{equation}\label{cd}
  \begin{split}
  & C_n := \begin{pmatrix} -\frac{a \mu}{n^2} & 1 & 1\\[1mm] -\frac{a\mu}{n} & i + \frac{a\mu^2}{2n} & -i + \frac{a\mu^2}{2n}
	\\[1mm]
	1- \frac{a\mu^2}{n^2} & -\frac{\mu i}{n} + \frac{\mu}{n^2} - \frac{a\mu^3}{2n^2} & \frac{\mu i}{n} + \frac{\mu}{n^2}
	- \frac{a\mu^3}{2n^2} \end{pmatrix}, \\
	&D_n := {\rm diag}\left(-n^2 + a\mu^2, -ni - \frac{a \mu^2}{2}, ni - \frac{a\mu^2}{2}\right).
	\end{split}
\end{equation}
We get
\begin{equation}\label{cinv}
  C_n^{-1} = \begin{pmatrix} -\frac{\mu}{n^2} +O(\frac{1}{n^4}) & \frac{\mu}{n} +O(\frac{1}{n^3}) & 1 + \frac{2a\mu^2}{n^2}
  +O(\frac{1}{n^4}) \\[1mm] \frac{1}{2} + \frac{a\mu^2 i}{4n} +O(\frac{1}{n^3}) & -\frac{i}{2} +O(\frac{1}{n^2})& -\frac{a\mu i}{2n}
  +O(\frac{1}{n^2}) \\[1mm] \frac{1}{2} - \frac{a\mu^2 i}{4n} +O(\frac{1}{n^3}) & \frac{i}{2} +O(\frac{1}{n^2}) & \frac{a\mu i}{2n}  +O(\frac{1}{n^2}) \end{pmatrix}
\end{equation}
and
$$C_n D_n = \begin{pmatrix} a\mu + O(\frac{1}{n}) & -n i - \frac{a\mu^2}{2} + O(\frac{1}{n}) & n i - \frac{a\mu^2}{2}
+ O(\frac{1}{n}) \\[1mm] a\mu n + O(\frac{1}{n}) & n - a \mu^2 i + O(\frac{1}{n}) & n + a \mu^2 i + O(\frac{1}{n}) \\[1mm]
-n^2 + 2 a \mu^2 + O(\frac{1}{n}) & - \mu + O(\frac{1}{n}) & - \mu + O(\frac{1}{n}) \end{pmatrix}$$
as well as
$$C_n D_n C_n^{-1} = \begin{pmatrix} O(\frac{1}{n}) & -n + O(\frac{1}{n}) & O(\frac{1}{n}) \\[1mm]
n + O(\frac{1}{n}) & O(\frac{1}{n}) & a \mu n + O(\frac{1}{n}) \\[1mm]
O(\frac{1}{n}) & - \mu n + O(\frac{1}{n}) & - n^2 + O(\frac{1}{n}) \end{pmatrix}.$$
Hence, we have
$$C_n D_n C_n^{-1} = A_{n,a}^\ast$$
up to a componentwise error of $O(\frac{1}{n})$ and therefore there is $c>0$ such that
\begin{equation}\label{pert1}
\| A_{n,a}^\ast - C_n D_n C_n^{-1}\| \le \frac{c}{n}, \qquad n \in \N.
\end{equation}

Next, we study the projections
\begin{equation}\label{proj}
  \begin{split}
  U_{1,n}(t) &:= \begin{pmatrix} -\frac{a \mu}{n^2} \\[1mm] -\frac{a\mu}{n} \\[1mm] 1- \frac{a\mu^2}{n^2} \end{pmatrix} \cdot
	\begin{pmatrix} n \hat{u}_n \\ \hat{v}_n \\ \hat{\theta}_n \end{pmatrix}(t), \\
	U_{2,n}(t) &:= \begin{pmatrix} 1 \\[1mm] i + \frac{a\mu^2}{2n}	\\[1mm] -\frac{\mu i}{n} + \frac{\mu}{n^2} - \frac{a\mu^3}{2n^2}
	\end{pmatrix} \cdot
	\begin{pmatrix} n \hat{u}_n \\ \hat{v}_n \\ \hat{\theta}_n \end{pmatrix}(t), \\
	U_{3,n}(t) &:= \begin{pmatrix} 1\\[1mm] -i + \frac{a\mu^2}{2n} \\[1mm] \frac{\mu i}{n} + \frac{\mu}{n^2} - \frac{a\mu^3}{2n^2}
	\end{pmatrix} \cdot
	\begin{pmatrix} n \hat{u}_n \\ \hat{v}_n \\ \hat{\theta}_n \end{pmatrix}(t),
	\end{split}
\end{equation}
of the solution to \eqref{nonlinsys} in direction of the eigenvectors $V_{j,n}$ defined above. In view of \eqref{nonlinsys} and
\eqref{nonlinsysf}, we get
\begin{align*}
  U_{1,n}^\prime &= -\frac{a \mu}{n} \, \hat{v}_n + a \mu n \hat{u}_n + a\mu^2 \hat{\theta}_n + \left(1- \frac{a\mu^2}{n^2}
	\right) a \mu n \hat{v}_n \\
	&\hspace*{+5mm} + (-n^2+a\mu^2) \hat{\theta}_n + \left(1- \frac{a\mu^2}{n^2} \right) g_{3,n} \\
	&= (-n^2+a\mu^2) U_{1,n} + \frac{a^2\mu^3}{n^2} \, (n \hat{u}_n) - \frac{a\mu}{n} \, \hat{v}_n + \frac{a^2\mu^4}{n^2} \,
	\hat{\theta}_n \\
	& \hspace*{+5mm} + \left(1- \frac{a\mu^2}{n^2} \right) g_{3,n}, \\
	U_{2,n}^\prime &= n \hat{v}_n + \left(-ni-\frac{a\mu^2}{2} \right) n \hat{u}_n + \left(-\mu n i - \frac{a\mu^3}{2} \right)
	\hat{\theta}_n \\
	&\hspace*{+5mm} + \left(-a\mu^2 i + \frac{a\mu^2}{n} - \frac{a^2\mu^4}{2n} \right) \hat{v}_n + \left(\mu n i - \mu
	+\frac{a\mu^3}{2} \right) \hat{\theta}_n \\
	&\hspace*{+5mm} + \left(- \frac{\mu i}{n} + \frac{\mu}{n^2} - \frac{a\mu^3}{2n^2}\right) g_{3,n} \\
	&= \left(-ni -\frac{a\mu^2}{2} \right) U_{2,n} + \left(\frac{a\mu^2}{n} - \frac{a^2\mu^4}{4n}\right) \hat{v}_n
	+\Big( \frac{\mu i}{n} - \frac{a\mu^3 i}{n} \\
	&\hspace*{+5mm} + \frac{a\mu^3}{2n^2} - \frac{a^2\mu^5}{4n^2} \Big) \hat{\theta}_n
	+ \left(- \frac{\mu i}{n} + \frac{\mu}{n^2} - \frac{a\mu^3}{2n^2}\right) g_{3,n}, \\
	U_{3,n}^\prime &= n \hat{v}_n + \left(ni-\frac{a\mu^2}{2} \right) n \hat{u}_n + \left(\mu n i - \frac{a\mu^3}{2} \right)
	\hat{\theta}_n \\
	&\hspace*{+5mm} + \left(a\mu^2 i + \frac{a\mu^2}{n} - \frac{a^2\mu^4}{2n} \right) \hat{v}_n + \left(-\mu n i - \mu
	+\frac{a\mu^3}{2} \right) \hat{\theta}_n \\
	&\hspace*{+5mm} + \left(\frac{\mu i}{n} + \frac{\mu}{n^2} - \frac{a\mu^3}{2n^2}\right) g_{3,n} \\
	&= \left(ni -\frac{a\mu^2}{2} \right) U_{3,n} + \left(\frac{a\mu^2}{n} - \frac{a^2\mu^4}{4n}\right) \hat{v}_n
	+\Big( -\frac{\mu i}{n} + \frac{a\mu^3 i}{n} \\
	&\hspace*{+5mm} + \frac{a\mu^3}{2n^2} - \frac{a^2\mu^5}{4n^2} \Big) \hat{\theta}_n
	+ \left(\frac{\mu i}{n} + \frac{\mu}{n^2} - \frac{a\mu^3}{2n^2}\right) g_{3,n} .
\end{align*}
Hence, we have for all $n \in \N$
\begin{equation}\label{sysproj}
  \begin{split}
	  U_{1,n}^\prime &= ({-}n^2+a\mu^2) U_{1,n} + F_{1,n}, \quad t \in (0,\infty), \\
		U_{2,n}^\prime &= \left({-}ni -\frac{a\mu^2}{2} \right) U_{2,n} + F_{2,n}, \quad t \in (0,\infty), \\
		U_{3,n}^\prime &= \left(ni -\frac{a\mu^2}{2} \right) U_{3,n} + F_{3,n}, \quad t \in (0,\infty),
	\end{split}
\end{equation}
with
\begin{align}\label{sysprojf}
    F_{1,n} &:= \frac{a^2\mu^3}{n^2} \, (n \hat{u}_n) - \frac{a\mu}{n} \, \hat{v}_n + \frac{a^2\mu^4}{n^2} \,
	  \hat{\theta}_n \nonumber \\
	  & \hspace*{+5mm} + \left(1- \frac{a\mu^2}{n^2} \right) \bigg(\frac{\mu}{2} \sum\limits_{k=1}^{n-1}
	  \hat{\theta}_{n-k} k \hat{v}_k + \frac{\mu}{2} \sum\limits_{l=1}^\infty \hat{\theta}_{l+n} l \hat{v}_l \nonumber \\
	  & \hspace*{+5mm} + \frac{\mu}{2} \sum\limits_{l=1}^\infty \hat{\theta}_l (l+n) \hat{v}_{l+n}
	  + \mu (\hat{\theta}_0 -a) n \hat{v}_n \bigg), \nonumber \\
		F_{2,n} &:= \left(\frac{a\mu^2}{n} - \frac{a^2\mu^4}{4n}\right) \hat{v}_n
	  +\Big( \frac{\mu i}{n} - \frac{a\mu^3 i}{n} + \frac{a\mu^3}{2n^2} - \frac{a^2\mu^5}{4n^2} \Big) \hat{\theta}_n \nonumber \\
	  &\hspace*{+5mm} + \left(- \frac{\mu i}{n} + \frac{\mu}{n^2} - \frac{a\mu^3}{2n^2}\right)
	  \bigg(\frac{\mu}{2} \sum\limits_{k=1}^{n-1} \hat{\theta}_{n-k} k \hat{v}_k \nonumber \\
	  & \hspace*{+5mm} + \frac{\mu}{2} \sum\limits_{l=1}^\infty \hat{\theta}_{l+n} l \hat{v}_l
	  + \frac{\mu}{2} \sum\limits_{l=1}^\infty \hat{\theta}_l (l+n) \hat{v}_{l+n}
		+ \mu (\hat{\theta}_0 -a) n \hat{v}_n \bigg), \nonumber \\
		F_{3,n} &:= \left(\frac{a\mu^2}{n} - \frac{a^2\mu^4}{4n}\right) \hat{v}_n
	  +\Big( {-}\frac{\mu i}{n} + \frac{a\mu^3 i}{n} + \frac{a\mu^3}{2n^2} - \frac{a^2\mu^5}{4n^2} \Big) \hat{\theta}_n \nonumber \\
	  &\hspace*{+5mm} + \left(\frac{\mu i}{n} + \frac{\mu}{n^2} - \frac{a\mu^3}{2n^2}\right)
		\bigg(\frac{\mu}{2} \sum\limits_{k=1}^{n-1}
	  \hat{\theta}_{n-k} k \hat{v}_k \nonumber \\
	  & \hspace*{+5mm} + \frac{\mu}{2} \sum\limits_{l=1}^\infty \hat{\theta}_{l+n} l \hat{v}_l
		+ \frac{\mu}{2} \sum\limits_{l=1}^\infty \hat{\theta}_l (l+n) \hat{v}_{l+n}
		+ \mu (\hat{\theta}_0 -a) n \hat{v}_n \bigg) .
\end{align}

\section{Estimates of nonlinear terms}\label{estimates}
Here we estimate the nonlinear terms appearing in \eqref{nonlinsysf} and \eqref{sysprojf} in some weighted $l^2$-spaces. These
estimates correspond, on the Fourier side, to the $H^s$ estimates on the original variables and will be used in the next section to set up an appropriate fixed point argument.

We first recall Young's inequality for the discrete convolution. For sequences $\mathcal{A}= (a_n)_{n \in \Z}$ and
$\mathcal{B}= (b_n)_{n \in \Z}$ the discrete convolution is
$$(\mathcal{A} \ast \mathcal{B})_n= \sum\limits_{k \in \Z} a_{n-k}b_k, \qquad n \in \Z.$$
Then, we have Young's inequality
$$\|\mathcal{A}\ast \mathcal{B}\|_{l^r} \le \|\mathcal{A}\|_{l^p} \|\mathcal{B}\|_{l^q} \quad\mbox{for } p,q,r \in [1,\infty],
1+ \frac{1}{r} = \frac{1}{p} + \frac{1}{q}.$$
Setting $a_n=b_n=0$ for $n\le 0$, we get in particular
\begin{equation}\label{eqconv1}
  \begin{split}
	\left\| \Big(\sum\limits_{k=1}^{n-1} a_{n-k}b_k \Big)_{n \in \N} \right\|_{l^r}
	&\le \|(a_n)_{n \in \N} \|_{l^p} \|(b_n)_{n \in \N} \|_{l^q} \\
	&\mbox{for } p,q,r \in [1,\infty], 1+ \frac{1}{r} = \frac{1}{p} + \frac{1}{q}.
	\end{split}
\end{equation}
Setting in addition $d_l = b_{-l}$, $l \in \Z$, we get for $n \in \N$ with $k=-l$
$$\sum\limits_{l=1}^\infty a_{l+n}b_l = \sum\limits_{k=-\infty}^{-1} a_{n-k} b_{-k}
= \sum\limits_{k=-\infty}^{-1} a_{n-k} d_k = \sum\limits_{k \in \Z} a_{n-k} d_k$$
and an application of Young's inequality implies
\begin{equation}\label{eqconv2}
  \begin{split}
	\left\| \Big(\sum\limits_{l=1}^\infty a_{l+n}b_l \Big)_{n \in \N} \right\|_{l^r}
	&\le \|(a_n)_{n \in \N} \|_{l^p} \|(b_n)_{n \in \N} \|_{l^q} \\
	&\mbox{for } p,q,r \in [1,\infty], 1+ \frac{1}{r} = \frac{1}{p} + \frac{1}{q}.
	\end{split}
\end{equation}
As a final preparation, we need to estimate a weighted $l^1$-norm in terms of a corresponding $l^2$-norm.
\begin{lem}\label{lem4.1}
  Let $s>0$, $\beta \ge 0$ such that $s> \beta + \frac{1}{2}$. Then, there exists $c=c(s,\beta)>0$ such that for any sequence
	$(a_n)_{n \in \N}$ we have
	$$\sum\limits_{n=1}^\infty n^\beta |a_n| \le c \left( \sum\limits_{n=1}^\infty n^{2s} |a_n|^2
	\right)^{\frac{1}{2}} .$$
\end{lem}
\begin{proof}
  By the Cauchy-Schwarz inequality in $l^2$ we have
	\begin{align*}
	  \sum\limits_{n=1}^\infty n^\beta |a_n|
		&= \sum\limits_{n=1}^\infty n^s |a_n| n^{\beta-s}
		\le \left(\sum\limits_{n=1}^\infty n^{2s} |a_n|^2 \right)^{\frac{1}{2}}
		\left( \sum\limits_{n=1}^\infty n^{2(\beta-s)} \right)^{\frac{1}{2}}.
	\end{align*}
	As $s> \beta + \frac{1}{2}$ implies $2(\beta -s) < -1$, the constant
	$c:= \left( \sum\limits_{n=1}^\infty n^{2(\beta-s)} \right)^{\frac{1}{2}}$ is finite. This proves the claim.
\end{proof}
Using these preparations, we get the following estimates for the nonlinear terms in \eqref{nonlinsysf} and \eqref{sysprojf}.
We notice that the assumption $s>\frac{3}{4}$ is crucial in the lemma below.
\begin{lem}\label{lem4.2}
  Let $s\in (\frac{3}{4},1)$ and $\alpha >0$. Then, there exists $c=c(s)>0$ such that
	\begin{align*}
	  &\hspace*{-5mm} \left( \sum\limits_{n=1}^\infty n^{2s} \sup\limits_{t>0} e^{2\alpha t} \left|\frac{1}{n}
		\sum\limits_{k=1}^{n-1} \hat{\theta}_{n-k}(t) k \hat{v}_k(t) \right|^2 \right)^{\frac{1}{2}} \\
		&\le c \left( \sum\limits_{n=1}^\infty n^{2s} \sup\limits_{t>0} e^{2\alpha t} |\hat{v}_n(t)|^2
		\right)^{\frac{1}{2}} \left( \sum\limits_{n=1}^\infty n^{2s} \sup\limits_{t>0} e^{2\alpha t} |\hat{\theta}_n(t)|^2
		\right)^{\frac{1}{2}}, \\
		&\hspace*{-5mm} \left( \sum\limits_{n=1}^\infty n^{2s} \sup\limits_{t>0} e^{2\alpha t} \left| \frac{1}{n}
		\sum\limits_{l=1}^\infty \hat{\theta}_{l+n}(t) l \hat{v}_l(t) \right|^2 \right)^{\frac{1}{2}} \\
		&\le c \left( \sum\limits_{n=1}^\infty n^{2s} \sup\limits_{t>0} e^{2\alpha t} |\hat{v}_n(t)|^2 \right)^{\frac{1}{2}}
		\left( \sum\limits_{n=1}^\infty n^{2s} \sup\limits_{t>0} e^{2\alpha t} |\hat{\theta}_n(t)|^2 \right)^{\frac{1}{2}}, \\
		&\hspace*{-5mm} \left( \sum\limits_{n=1}^\infty n^{2s} \sup\limits_{t>0} e^{2\alpha t} \left| \frac{1}{n}
		\sum\limits_{l=1}^\infty \hat{\theta}_l(t) (l+n) \hat{v}_{l+n}(t) \right|^2	\right)^{\frac{1}{2}} \\
		&\le c \left( \sum\limits_{n=1}^\infty n^{2s} \sup\limits_{t>0} e^{2\alpha t} |\hat{v}_n(t)|^2 \right)^{\frac{1}{2}}
		\left( \sum\limits_{n=1}^\infty n^{2s} \sup\limits_{t>0} e^{2\alpha t} |\hat{\theta}_n(t)|^2 \right)^{\frac{1}{2}} .
	\end{align*}
\end{lem}
\begin{proof}
  Using \eqref{eqconv1} with $r=q=2$, $p=1$, $a_k:= \sup\limits_{t>0} e^{\alpha t} |\hat{\theta}_k(t)|$, and
	$b_k:= k^s \sup\limits_{t>0} e^{\alpha t} |\hat{v}_k(t)|$ as well as Lemma~\ref{lem4.1} with $\beta=0$, we have due to $s<1$,
	$\alpha >0$, and the properties of the supremum of nonnegative functions
	\begin{align*}
	  &\hspace*{-5mm} \left( \sum\limits_{n=1}^\infty n^{2s} \sup\limits_{t>0} e^{2\alpha t} \left|\frac{1}{n}
		\sum\limits_{k=1}^{n-1} \hat{\theta}_{n-k}(t) k \hat{v}_k(t) \right|^2 \right)^{\frac{1}{2}} \\
		&= \left( \sum\limits_{n=1}^\infty \sup\limits_{t>0} \left| \sum\limits_{k=1}^{n-1} \left(\frac{k}{n}\right)^{1-s} k^s
		e^{\alpha t}\hat{v}_k(t) \hat{\theta}_{n-k}(t) \right|^2 \right)^{\frac{1}{2}} \\
		&\le \left( \sum\limits_{n=1}^\infty \left|\sup\limits_{t>0} \sum\limits_{k=1}^{n-1} k^s e^{\alpha t}|\hat{v}_k(t)| e^{\alpha t}
		|\hat{\theta}_{n-k}(t)| \right|^2 \right)^{\frac{1}{2}} \\
		&\le \left( \sum\limits_{n=1}^\infty \left| \sum\limits_{k=1}^{n-1} \Big(k^s \sup\limits_{t>0} e^{\alpha t}|\hat{v}_k(t)|\Big)
		\Big( \sup\limits_{t>0}e^{\alpha t}|\hat{\theta}_{n-k}(t)| \Big)\right|^2 \right)^{\frac{1}{2}} \\
		&\le \left(\sum\limits_{n=1}^\infty n^{2s} \sup\limits_{t>0} e^{2\alpha t}|\hat{v}_n(t)|^2 \right)^{\frac{1}{2}}
		\left( \sum\limits_{n=1}^\infty \sup\limits_{t>0} e^{\alpha t} |\hat{\theta}_n(t)| \right) \\
		&\le c \left( \sum\limits_{n=1}^\infty n^{2s} \sup\limits_{t>0} e^{2\alpha t}|\hat{v}_n(t)|^2 \right)^{\frac{1}{2}}
		\left( \sum\limits_{n=1}^\infty n^{2s} \sup\limits_{t>0} e^{2\alpha t}|\hat{\theta}_n(t)|^2 \right)^{\frac{1}{2}},
	\end{align*}
	since $\beta + \frac{1}{2} = \frac{1}{2} < s$.
	
	Next, in view of $s \in (\frac{3}{4},1)$, we may fix $\eps >0$ such that $2s > \frac{3}{2} + \eps$ and $s > \frac{1}{2} + \eps$.
	Using that $(l+n)^{-s+\frac{1}{2} +\eps} \le l^{-s+\frac{1}{2} +\eps}$, $n^{s-1} \le 1$, and $l^{-2s+\frac{3}{2}+\eps} \le 1$
	for $l,n \in \N$, we get from
	\eqref{eqconv2} with $r=q=2$, $p=1$, $a_l := l^{s-\frac{1}{2}-\eps} \sup\limits_{t>0} e^{\alpha t} |\hat{\theta}_l(t)|$, and
	$b_l := l^s \sup\limits_{t>0} e^{\alpha t} |\hat{v}_l(t)|$ as well as Lemma~\ref{lem4.1} with $\beta=s-\frac{1}{2}-\eps$
	and the properties of the supremum of nonnegative functions and $\alpha >0$
	\begin{align*}
	  &\hspace*{-2mm}\left( \sum\limits_{n=1}^\infty n^{2s} \sup\limits_{t>0} e^{2\alpha t} \left| \frac{1}{n}
		\sum\limits_{l=1}^\infty \hat{\theta}_{l+n}(t) l \hat{v}_l(t) \right|^2 \right)^{\frac{1}{2}} \\
		&= \left( \sum\limits_{n=1}^\infty \sup\limits_{t>0} \left|n^{s-1} \sum\limits_{l=1}^\infty l^{1-s} l^s e^{\alpha t}
		\hat{v}_l (t)	(l+n)^{s-\frac{1}{2}-\eps} \hat{\theta}_{l+n}(t) (l+n)^{-s+\frac{1}{2} +\eps} \right|^2 \right)^{\frac{1}{2}} \\
		&\le \left( \sum\limits_{n=1}^\infty \left| \sup\limits_{t>0} \sum\limits_{l=1}^\infty l^{-2s+\frac{3}{2}+\eps} l^s
		e^{\alpha t} |\hat{v}_l(t)|
		(l+n)^{s-\frac{1}{2}-\eps}e^{\alpha t}|\hat{\theta}_{l+n}(t)| \right|^2 \right)^{\frac{1}{2}} \\
		&\le \left( \sum\limits_{n=1}^\infty \left| \sum\limits_{l=1}^\infty \Big( l^s \sup\limits_{t>0} e^{\alpha t} |\hat{v}_l(t)|
		\Big) \Big((l+n)^{s-\frac{1}{2}-\eps} \sup\limits_{t>0} e^{\alpha t} |\hat{\theta}_{l+n}(t)| \Big) \right|^2
		\right)^{\frac{1}{2}} \\
		&\le \left(\sum\limits_{n=1}^\infty n^{2s} \sup\limits_{t>0} e^{2\alpha t} |\hat{v}_n(t)|^2 \right)^{\frac{1}{2}}
		\left( \sum\limits_{n=1}^\infty n^{s-\frac{1}{2}-\eps} \sup\limits_{t>0} e^{\alpha t}|\hat{\theta}_n(t)|
		\right) \\
		&\le c \left( \sum\limits_{n=1}^\infty n^{2s} \sup\limits_{t>0} e^{2\alpha t} |\hat{v}_n(t)|^2 \right)^{\frac{1}{2}}
		\left( \sum\limits_{n=1}^\infty n^{2s} \sup\limits_{t>0} e^{2\alpha t} |\hat{\theta}_n(t)|^2 \right)^{\frac{1}{2}},
	\end{align*}
	as $\beta +\frac{1}{2} = s-\eps <s$.
	
	Similarly, using that $n^{s-1} \le 1$, $l^{s-\frac{1}{2}-\eps} \ge 1$, and $l^{-2s+\frac{3}{2}+\eps} \le 1$
	for $l,n \in \N$, we obtain from \eqref{eqconv2} with $r=p=2$ and $q=1$ as well as Lemma~\ref{lem4.1} with
	$\beta=s-\frac{1}{2}-\eps$
	\begin{align*}
	  &\hspace*{-3mm} \left( \sum\limits_{n=1}^\infty n^{2s} \sup\limits_{t>0} e^{2\alpha t} \left| \frac{1}{n}
		\sum\limits_{l=1}^\infty \hat{\theta}_l(t) (l+n) \hat{v}_{l+n}(t) \right|^2	\right)^{\frac{1}{2}} \\
		&= \bigg(\sum\limits_{n=1}^\infty \sup\limits_{t>0} \Big| n^{s-1} \sum\limits_{l=1}^n (l+n)^{1-s} (l+n)^s e^{\alpha t}
		\hat{v}_{l+n}(t) \hat{\theta}_l(t) \\
		&\hspace*{+5mm} + n^{s-1} \sum\limits_{l=n+1}^\infty (l+n)^{1-s} (l+n)^s e^{\alpha t} \hat{v}_{l+n}(t) l^{s-\frac{1}{2}-\eps}
		\hat{\theta}_l (t) l^{-s+\frac{1}{2}+\eps}\Big|^2 \bigg)^{\frac{1}{2}} \\
		&\le \bigg( \sum\limits_{n=1}^\infty \sup\limits_{t>0} \Big| \left(\frac{2n}{n} \right)^{1-s} \sum\limits_{l=1}^n (l+n)^s
		e^{\alpha t} |\hat{v}_{l+n}(t)| e^{\alpha t}|\hat{\theta}_l(t)| \\
		&\hspace*{+5mm} + \sum\limits_{l=n+1}^\infty (2l)^{1-s} (l+n)^s e^{\alpha t}|\hat{v}_{l+n}(t)| l^{s-\frac{1}{2}-\eps}
		e^{\alpha t}|\hat{\theta}_l (t)| l^{-s+\frac{1}{2}+\eps}\Big|^2 \bigg)^{\frac{1}{2}} \\
		&\le \bigg( \sum\limits_{n=1}^\infty \Big| 2^{1-s} \sum\limits_{l=1}^n \Big((l+n)^s \sup\limits_{t>0} e^{\alpha t}
		|\hat{v}_{l+n}(t)| \Big) \Big(l^{s-\frac{1}{2}-\eps} \sup\limits_{t>0} e^{\alpha t} |\hat{\theta}_l(t)|\Big)  \\
		&\hspace*{+5mm} + 2^{1-s} \sum\limits_{l=n+1}^\infty l^{-2s+\frac{3}{2}+\eps} \Big( (l+n)^s \sup\limits_{t>0} e^{\alpha t}
		|\hat{v}_{l+n}(t)| \Big) \cdot \\
		&\hspace*{+5mm} \Big( l^{s-\frac{1}{2}-\eps} \sup\limits_{t>0} e^{\alpha t}
		|\hat{\theta}_l(t)|\Big) \Big|^2 \bigg)^{\frac{1}{2}} \\
		&\le 2^{1-s} \left( \sum\limits_{n=1}^\infty \left| \sum\limits_{l=1}^\infty \Big( (l+n)^s \sup\limits_{t>0} e^{\alpha t}
		|\hat{v}_{l+n}(t)|\Big) \Big(	l^{s-\frac{1}{2}-\eps} \sup\limits_{t>0} e^{\alpha t} |\hat{\theta}_l(t)|\Big) \right|^2
		\right)^{\frac{1}{2}} \\
		&\le 2^{1-s} \left(\sum\limits_{n=1}^\infty n^{2s} \sup\limits_{t>0} e^{2\alpha t}|\hat{v}_n(t)|^2 \right)^{\frac{1}{2}}
		\left( \sum\limits_{n=1}^\infty n^{s-\frac{1}{2}-\eps} \sup\limits_{t>0} e^{\alpha t}|\hat{\theta}_n(t)|
		\right) \\
		&\le 2^{1-s} c \left( \sum\limits_{n=1}^\infty n^{2s} \sup\limits_{t>0} e^{2\alpha t} |\hat{v}_n(t)|^2 \right)^{\frac{1}{2}}
		\left( \sum\limits_{n=1}^\infty n^{2s} \sup\limits_{t>0} e^{2\alpha t} |\hat{\theta}_n(t)|^2 \right)^{\frac{1}{2}},
	\end{align*}
	as $\beta +\frac{1}{2} = s-\eps <s$. Hence, the claim is proved.
\end{proof}

\section{Fixed point argument}\label{contraction}
Here we use a fixed point argument to prove the existence of a sequence of solutions to \eqref{nonlinsys}, which for large $n$ is
equivalent to \eqref{sysproj}.
We choose $a:= \theta_\infty$, with $\theta_\infty>0$ defined in \eqref{theta_in}, and let $\lambda_{j,n}$, $j=1,2,3$, denote the
eigenvalues of $A_{n,\theta_\infty}$ (which coincide with the eigenvalues of $A_{n, \theta_\infty}^\ast$). Then in view of
Lemma~\ref{lem3.3}, with $n_0$ from that lemma we fix $N_0 \in \N$ such that
\begin{equation}\label{n0}
  \begin{split}
    & N_0 \ge \max \left\{ n_0, (2 \mu^2 \theta_\infty)^{\frac{1}{2}}, 72 (1+\theta_\infty^2)(1+ \mu^4), \frac{576 (1+\theta_\infty^2)(1+ \mu^4)}{\theta_\infty\mu} \right\}, \\
	  &\RRe(\lambda_{j,n}) \le - \frac{\mu^2 \theta_\infty}{4} \quad\mbox{for all } n \ge N_0, \, j=1,2,3, \quad\mbox{and}\\
		& |(C_n)_{k,l}| \le 2, \, |(C_n^{-1})_{k,l}| \le 2 \quad\mbox{for all } n \ge N_0, \, k,l = 1,2,3,
	\end{split}
\end{equation}
where $C_n$ and $C_n^{-1}$ are defined in \eqref{cd} and \eqref{cinv}.
Then in view of Lemma~\ref{lem3.1} we choose $\alpha>0$ such that
\begin{equation}\label{alpha}
\begin{split}
& \alpha := \min\{\alpha_1, \alpha_2\} \quad\mbox{with}\quad \alpha_2:=\frac{\mu^2\theta_\infty}{4}, \\
& \alpha_1 :=\frac{1}{3}\min\{ -\RRe(\lambda_{j,n}) \: : \: j=1,2,3, n<N_0\} .
\end{split}
\end{equation}
Next, for $s \in (0,1)$ we define the following norm for the sequence $z (t) = (z_n(t))_{n \in \N}$:
\[
|z|_s=\sqrt{\sum_{n=1}^\infty n^{2s} \sup_{t>0} e^{2\alpha t} |z_n(t)|^2}.
\]
Next, we define the space $X$, consisting of all the vectors $(U_{1}, U_2, U_3, \hat{\theta}_0)$, where
$U_{j} = (U_{j,n})_{n \in \N}$ are sequences of complex valued functions $U_{j,n}:[0,\infty) \to \C$,
$\hat{\theta}_0:[0,\infty)\rightarrow \C$, with a finite norm
\begin{equation}\label{pX}
\|(U_1,U_2,U_3,\hat{\theta}_0)\|_{X}:=\max\{|U_j|_s, j=1,2,3,\sup_{t>0}|\hat{\theta}_0(t)|\}.
\end{equation}
Next, consider $\overline{B((0,0,0,\theta_\infty),R)}$ of radius $R$ in $X$, centered at $(0,0,0,\theta_\infty)$.

Before proceeding with a Banach contraction argument in $X$, let us state the following preparatory lemma.
\begin{lem}\label{calka}
Let $\beta, \gamma\in \r$ such that $\beta > \gamma$. For $f:[0,\infty)\rightarrow \r$ consider the function $g$ given as
\[
g(t):=\int_0^t e^{-\beta(t-\sigma)}|f(\sigma)|\;d\sigma.
\]
Then,
\begin{equation}\label{pomocnicze}
\sup_{t>0}e^{\gamma t}|g(t)|\leq \frac{1}{\beta-\gamma}\, \sup_{t>0}e^{\gamma t} |f(t)|.
\end{equation}
\end{lem}
\begin{proof}
We calculate
\[
e^{\gamma t} |g(t)|\leq e^{\gamma t}\int_0^t e^{-\beta(t-\sigma)}|f(\sigma)|\;d\sigma
\leq c_1e^{\gamma t}\int_0^te^{-\gamma \sigma} e^{-\beta(t-\sigma)}d\sigma,
\]
where $c_1:=\sup_{t>0}e^{\gamma t}|f(t)|$. Next, since
\[
\int_0^te^{(\beta-\gamma)\sigma}d\sigma=\frac{1}{\beta-\gamma}\left(e^{(\beta-\gamma)t}-1\right),
\]
we observe that
\[
e^{\gamma t} |g(t)|\leq \frac{c_1}{\beta-\gamma}\left(1- e^{-(\beta-\gamma)t}\right)
\leq \frac{c_1}{\beta-\gamma},
\]
and the lemma is proven.
\end{proof}
We are now in a position to define a mapping
\[
\mathcal{F}:\overline{B((0,0,0,\theta_\infty),R)}\subset X\rightarrow \overline{B((0,0,0,\theta_\infty),R)}
\]
that is a contraction. We shall next apply the Banach contraction mapping theorem to our mapping $\mathcal{F}$. We first fix initial data $(U_1(0),U_2(0),U_3(0)$, $\hat{\theta}_0(0))$ such that
\begin{equation}\label{eqini}
    \max\left\{ \left(\sum\limits_{n=1}^\infty n^{2s} |U_{j,n}(0)|^2
	  \right)^{\frac{1}{2}}, \, j=1,2,3, \, |\hat{\theta}_0(0)-\theta_\infty| \right\} \le \delta R
\end{equation}
with some $\delta \in (0, \frac{1}{4}]$ to be fixed below. We then define $\mathcal{F}:\overline{B((0,0,0,\theta_\infty),R)}$
$\subset X\rightarrow X
$,
\begin{equation}\label{kontrakcja}
\mathcal{F}(A_1, A_2, A_3, h)=(\Psi_1, \Psi_2, \Psi_3, \xi),
\end{equation}
in the following way: We first take $n \ge N_0$. Then in view of \eqref{proj} we have
$(U_{1,n}, U_{2,n}, U_{3,n})^T = C_n^T (n\hat{u}_n, \hat{v}_n, \hat{\theta}_n)^T$, where $C_n$ is defined in \eqref{cd}.
Hence, with $B_n := C_n^T$ we have $B_n^{-1} = (C_n^{-1})^T$ with $C_n^{-1}$ from \eqref{cinv} and we denote
$(a_{1,n}, a_{2,n}, a_{3,n})^T= B_n^{-1} (A_{1,n}, A_{2,n}, A_{3,n})^T$. In view of \eqref{sysproj}, which is equivalent to
\eqref{nonlinsys} for such $n$, and \eqref{f2} for $\hat{\theta}_0$, we define
\begin{align}
\Psi_{1,n}(t)&:=e^{(-n^2+\theta_\infty\mu^2) t} U_{1,n}(0)+\int_0^te^{(-n^2+\theta_\infty \mu^2)(t-\sigma)}F_{1,n}(\sigma)d\sigma,
\label{kontr1}\\
\Psi_{2,n}(t)&:=e^{(-ni-\frac{\theta_\infty\mu^2}{2}) t} U_{2,n}(0)+\int_0^t e^{(-ni-\frac{\theta_\infty\mu^2}{2})(t-\sigma)}
F_{2,n}(\sigma)d\sigma, \label{kontr2}\\
\Psi_{3,n}(t)&:=e^{(ni-\frac{\theta_\infty\mu^2}{2}) t} U_{3,n}(0)+\int_0^t e^{(ni-\frac{\theta_\infty\mu^2}{2})(t-\sigma)}
F_{3,n}(\sigma)d\sigma, \label{kontr3}\\
\xi(t) &:= \hat{\theta}_0(0) + \frac{\mu}{2} \int_0^t \sum_{l=1}^\infty f_l(\sigma) d\sigma \label{kontr4},
\end{align}
where for $l \in \N$ and $j=1,2,3$ we set
\begin{align}\label{F}
  f_l&:=a_{3,l}la_{2,l} \nonumber \\
	F_{1,n} &:= \frac{\theta_\infty^2\mu^3}{n^2} \, a_{1,n} - \frac{\theta_\infty\mu}{n} \, a_{2,n}
	+ \frac{\theta_\infty^2\mu^4}{n^2} \, a_{3,n} \nonumber \\
	& \hspace*{+5mm} + \left(1- \frac{\theta_\infty\mu^2}{n^2} \right) \bigg(\frac{\mu}{2} \sum\limits_{k=1}^{n-1}
	  a_{3,n-k} k a_{2,k} + \frac{\mu}{2} \sum\limits_{l=1}^\infty a_{3,l+n} l a_{2,l} \nonumber \\
	  & \hspace*{+5mm} + \frac{\mu}{2} \sum\limits_{l=1}^\infty a_{3,l} (l+n) a_{2,l+n}
	  + \mu (h-\theta_\infty) n a_{2,n} \bigg), \nonumber \\
		F_{2,n} &:= \left(\frac{\theta_\infty \mu^2}{n} - \frac{\theta_\infty^2\mu^4}{4n}\right) a_{2,n}
	  +\Big( \frac{\mu i}{n} - \frac{\theta_\infty\mu^3 i}{n} + \frac{\theta_\infty\mu^3}{2n^2} - \frac{\theta_\infty^2\mu^5}{4n^2}
		\Big) a_{3,n} \nonumber \\
	  &\hspace*{+5mm} + \left(- \frac{\mu i}{n} + \frac{\mu}{n^2} - \frac{\theta_\infty\mu^3}{2n^2}\right)
	  \bigg(\frac{\mu}{2} \sum\limits_{k=1}^{n-1} a_{3,n-k} k a_{2,k} \nonumber \\
	  & \hspace*{+5mm} + \frac{\mu}{2} \sum\limits_{l=1}^\infty a_{3,l+n} l a_{2,l}
	  + \frac{\mu}{2} \sum\limits_{l=1}^\infty a_{3,l} (l+n) a_{2,l+n}
		+ \mu (h-\theta_\infty) n a_{2,n} \bigg), \nonumber \\
		F_{3,n} &:= \left(\frac{\theta_\infty\mu^2}{n} - \frac{\theta_\infty^2\mu^4}{4n}\right) a_{2,n}
	  +\Big({-}\frac{\mu i}{n} + \frac{\theta_\infty\mu^3 i}{n} + \frac{\theta_\infty\mu^3}{2n^2} - \frac{\theta_\infty^2\mu^5}{4n^2}
		\Big) a_{3,n} \nonumber \\
	  &\hspace*{+5mm} + \left(\frac{\mu i}{n} + \frac{\mu}{n^2} - \frac{\theta_\infty\mu^3}{2n^2}\right)
		\bigg(\frac{\mu}{2} \sum\limits_{k=1}^{n-1}
	  a_{3,n-k} k a_{2,k} \nonumber \\
	  & \hspace*{+5mm} + \frac{\mu}{2} \sum\limits_{l=1}^\infty a_{3,l+n} l a_{2,l}
		+ \frac{\mu}{2} \sum\limits_{l=1}^\infty a_{3,l} (l+n) a_{2,l+n}
		+ \mu (h-\theta_\infty) n a_{2,n} \bigg) ,
\end{align}
where $a_{j,n}$ for $j=1,2,3$ and $n <N_0$ are defined below.

We next take $n < N_0$. As we do not know whether $A_{n,\theta_\infty}$ from \eqref{linsysa} is diagonalizable for all such $n$,
let $B_n \in \C^{3 \times 3}$ be chosen such that the columns of $B_n^T$ form a Jordan basis of $A_{n,\theta_\infty}$ and
$A_{n,\theta_\infty} = B_n^T J_n (B_n^T)^{-1}$ with $J_n$ being the Jordan normal form of $A_{n,\theta_\infty}$. Then, similarly
to \eqref{proj}, $(U_{1,n}, U_{2,n}, U_{3,n})^T = B_n (n\hat{u}_n, \hat{v}_n, \hat{\theta}_n)^T$ are the projections with respect
to the Jordan basis and we denote $(a_{1,n}, a_{2,n}, a_{3,n})^T= B_n^{-1} (A_{1,n}, A_{2,n}, A_{3,n})^T$. In view of
\eqref{nonlinsys}, we then define
\begin{align}
&\begin{pmatrix} \Psi_{1,n}(t) \\ \Psi_{2,n}(t) \\ \Psi_{3,n}(t) \end{pmatrix}
:= B_n \left( e^{t A_{n,\theta_\infty}} B_n^{-1}\begin{pmatrix} U_{1,n}(0) \\ U_{2,n}(0) \\ U_{3,n}(0) \end{pmatrix}
+\int_0^t e^{(t-\sigma)A_{n,\theta_\infty}} g_n(\sigma) d\sigma \right),
\label{kontr1a} 
\end{align}
with 
\begin{align}\label{g}
  g_n &:= \begin{pmatrix} 0 \\ 0 \\ g_{3,n} \end{pmatrix}, \qquad\mbox{where} \nonumber \\
  g_{3,n} &:=	\frac{\mu}{2} \sum\limits_{k=1}^{n-1}
	a_{3,n-k} k a_{2,k} + \frac{\mu}{2} \sum\limits_{l=1}^\infty a_{3,l+n} l a_{2,l} \nonumber \\
	& \hspace*{+5mm} + \frac{\mu}{2} \sum\limits_{l=1}^\infty a_{3,l} (l+n) a_{2,l+n}
	+ \mu (h-\theta_\infty) n a_{2,n} .
\end{align}
We are now in the position to prove that $\mathcal{F}$ has a fixed point.
\begin{prop}\label{prop5.2}
  Let $N_0 \in \N$ and $\alpha >0$ be as defined in \eqref{n0} and \eqref{alpha}, $\theta_\infty >0$ be as defined in \eqref{theta_in},
	and $s \in (\frac{3}{4},1)$.
	Then there are $R>0$ and $\delta \in (0,\frac{1}{4}]$ such that, for any initial data $(U_1(0),U_2(0),U_3(0),\hat{\theta}_0(0))$
	satisfying \eqref{eqini}, the mapping $\mathcal{F} : \overline{B((0,0,0,\theta_\infty),R)} \subset X \to X$ defined in \eqref{kontrakcja}--\eqref{g}
	has a unique fixed point $(U_1,U_2,U_3,\hat{\theta}_0) \in \overline{B((0,0,0,\theta_\infty),R)}$.
\end{prop}
\begin{proof}
  Clearly $\overline{B((0,0,0,\theta_\infty),R)}$ is a nonempty and closed subset of $X$. We next show that $\mathcal{F}$ is a self mapping and a
	contraction if $R$ and $\delta$ are chosen appropriately.
	
	\textbf{Step 1: Self mapping of a ball} \\
	Let $(A_1,A_2,A_3,h) \in \overline{B((0,0,0,\theta_\infty),R)}$. We aim to show $(\Psi_1, \Psi_2, \Psi_3, \xi) \in \overline{B((0,0,0,\theta_\infty),R)}$.
	We first consider $\Psi_1$ and $n \ge N_0$. In view of \eqref{n0} and \eqref{alpha}, we have $n^2 \ge 2 \theta_\infty \mu^2$ and
	hence $-n^2 + \theta_\infty \mu^2 \le -\frac{n^2}{2}$ and $-n^2 + \theta_\infty \mu^2 \le - \frac{\theta_\infty \mu^2}{2}
	\le -2\alpha$ and $\frac{n^2}{2} \ge \theta_\infty \mu^2 \ge 4 \alpha >0$, which implies $\frac{n^2}{2}-\alpha \ge
	\frac{n^2}{2}-\frac{n^2}{8} = \frac{3n^2}{8} \ge n$, since $N_0 \ge 72$. Thus, due to \eqref{kontr1} and Lemma~\ref{calka}, we
	have
	\begin{align}\label{eqsm1}
	  \sup\limits_{t>0} e^{\alpha t} |\Psi_{1,n}(t)|
		&\le \sup\limits_{t>0} e^{\alpha t} \Big( e^{(-n^2+\theta_\infty\mu^2) t} |U_{1,n}(0)| \nonumber \\
		&\hspace*{+5mm} +\int_0^t e^{(-n^2+\theta_\infty \mu^2)(t-\sigma)}|F_{1,n}(\sigma)|d\sigma \Big) \nonumber \\
		&\le |U_{1,n}(0)| + \sup\limits_{t>0} e^{\alpha t} \int_0^t e^{-\frac{n^2}{2}(t-\sigma)}|F_{1,n}(\sigma)|d\sigma \nonumber \\
		&\le |U_{1,n}(0)| + \frac{1}{\frac{n^2}{2}-\alpha} \, \sup\limits_{t>0} e^{\alpha t} |F_{1,n}(t)| \nonumber \\
		&\le |U_{1,n}(0)| + \frac{1}{n} \, \sup\limits_{t>0} e^{\alpha t} |F_{1,n}(t)|, \qquad n \ge N_0.
  \end{align}
	In view of $(a_{1,n}, a_{2,n}, a_{3,n})^T= B_n^{-1} (A_{1,n}, A_{2,n}, A_{3,n})^T$ and $B_n^{-1} = (C_n^{-1})^T$, \eqref{n0}
	implies $|(B_n^{-1})_{k,l}| \le 2$ and
	\begin{equation}\label{eqsm2}
	  |a_{j,n}| \le 2 (|A_{1,n}| + |A_{2,n}| + |A_{3,n}|), \qquad n \ge N_0, \, j=1,2,3 .
	\end{equation}
	Inserting this into \eqref{F} and using \eqref{n0}, we get
	\begin{align}\label{eqsm3}
	  \frac{1}{n} \, |F_{1,n}|
		&\le \frac{6 (1+\theta_\infty^2)(1+ \mu^4)}{n}(|A_{1,n}| + |A_{2,n}| + |A_{3,n}|) \nonumber \\
		&\hspace*{+5mm} + \frac{\mu(1+\theta_\infty \mu^2)}{2} \bigg( \frac{1}{n}\sum\limits_{k=1}^{n-1}
	  |a_{3,n-k}| k |a_{2,k}| + \frac{1}{n} \sum\limits_{l=1}^\infty |a_{3,l+n}| l |a_{2,l}|  \nonumber \\
	  & \hspace*{+5mm} + \frac{1}{n} \sum\limits_{l=1}^\infty |a_{3,l}| (l+n) |a_{2,l+n}| \bigg) \nonumber \\
	  & \hspace*{+5mm} + 2\mu (1+\theta_\infty \mu^2) |h-\theta_\infty| (|A_{1,n}| + |A_{2,n}| + |A_{3,n}|) \nonumber \\
		&\le \frac{1}{12}(|A_{1,n}| + |A_{2,n}| + |A_{3,n}|) \nonumber \\
		&\hspace*{+5mm} + \frac{\mu(1+\theta_\infty \mu^2)}{2} \bigg( \frac{1}{n}\sum\limits_{k=1}^{n-1}
	  |a_{3,n-k}| k |a_{2,k}| + \frac{1}{n} \sum\limits_{l=1}^\infty |a_{3,l+n}| l |a_{2,l}|  \nonumber \\
	  &\hspace*{+5mm} + \frac{1}{n} \sum\limits_{l=1}^\infty |a_{3,l}| (l+n) |a_{2,l+n}| \bigg) \nonumber \\
		&\hspace*{+5mm} + 2\mu (1+\theta_\infty \mu^2) R (|A_{1,n}| + |A_{2,n}| + |A_{3,n}|) .
	\end{align}
	Next, let $n < N_0$. As \eqref{alpha} implies that $\RRe(\lambda_{j,n}) \le -3\alpha$ for $j=1,2,3$, and $A_{n,\theta_\infty}$
	is a $(3 \times 3)$-matrix, there are $\tilde{c}(n), c(n)>0$ such that $|e^{t A_{n,\theta_\infty}} z| \le \tilde{c}(n)
	(1+t^2)e^{-3\alpha t} |z| \le c(n) e^{-2\alpha t}|z|$ for all $t\ge 0$ and all $z \in \R^3$. Due to the equivalence of norms
	in a finite dimensional vector space, there is $c_1 \ge 1$ such that
	\begin{equation}\label{eqsm4}
	  \begin{split}
	   &|(B_n e^{t A_{n,\theta_\infty}} B_n^{-1} z)_j|+  |(B_n e^{t A_{n,\theta_\infty}} z)_j|
		\le c_1 e^{-2\alpha t} (|z_1|+|z_2|+|z_3|) \\
		& \mbox{for all } t \ge 0, \, j=1,2,3 \mbox{ and all } n < N_0 .
		\end{split}
	\end{equation}
	Hence, \eqref{kontr1a} and Lemma~\ref{calka} imply for $j = 1,2,3$ and $n < N_0$
	\begin{align}\label{eqsm5}
	  \sup\limits_{t>0} e^{\alpha t} |\Psi_{j,n}(t)|
		&\le \sup\limits_{t>0} e^{\alpha t} \Big( c_1 e^{-2\alpha t}  (|U_{1,n}(0)|+ |U_{2,n}(0)| + |U_{3,n}(0)|) \nonumber \\
		&\hspace*{+5mm} + c_1 \int_0^t e^{-2\alpha(t-\sigma)}|g_{3,n}(\sigma)|d\sigma \Big) \nonumber \\
		&\le c_1 (|U_{1,n}(0)|+ |U_{2,n}(0)| + |U_{3,n}(0)|) \nonumber \\
		&\hspace*{+5mm} + c_1 \sup\limits_{t>0} e^{\alpha t} \int_0^t e^{-2\alpha(t-\sigma)}|g_{3,n}(\sigma)|d\sigma \nonumber \\
		&\le c_1 (|U_{1,n}(0)|+ |U_{2,n}(0)| + |U_{3,n}(0)|) \nonumber \\
		&\hspace*{+5mm} + \frac{c_1}{2\alpha -\alpha} \, \sup\limits_{t>0} e^{\alpha t} |g_{3,n}(t)| \nonumber \\
		&\le c_1 (|U_{1,n}(0)|+ |U_{2,n}(0)| + |U_{3,n}(0)|) \nonumber \\
		&\hspace*{+5mm} + \frac{c_1}{\alpha} \, \sup\limits_{t>0} e^{\alpha t} |g_{3,n}(t)|, \quad n < N_0 .
	\end{align}
	For any $n < N_0$ we have $(a_{1,n}, a_{2,n}, a_{3,n})^T= B_n^{-1} (A_{1,n}, A_{2,n}, A_{3,n})^T$. Since there is
	$c_2 \ge 2$ such that $|(B_n^{-1})_{k,l}| \le c_2$ for all $n < N_0$ and all $k,l = 1,2,3$, we obtain from \eqref{eqsm2}
	\begin{equation}\label{eqsm2a}
	  |a_{j,n}| \le c_2 (|A_{1,n}| + |A_{2,n}| + |A_{3,n}|), \qquad n \in \N, \, j=1,2,3 .
	\end{equation}
	By \eqref{g} and \eqref{eqsm2a}, we further have
	\begin{align}\label{eqsm6}
	  \frac{c_1}{\alpha} \, |g_{3,n}| &\le \frac{c_1 \mu N_0}{2\alpha} \bigg( \frac{1}{n}\sum\limits_{k=1}^{n-1}
	  |a_{3,n-k}| k |a_{2,k}| + \frac{1}{n} \sum\limits_{l=1}^\infty |a_{3,l+n}| l |a_{2,l}|  \nonumber \\
	  & \hspace*{+5mm} + \frac{1}{n} \sum\limits_{l=1}^\infty |a_{3,l}| (l+n) |a_{2,l+n}| \bigg)  \nonumber \\
		& \hspace*{+5mm} + \frac{c_1 c_2 \mu N_0 R}{\alpha} (|A_{1,n}| + |A_{2,n}| + |A_{3,n}|), \quad n < N_0 .
	\end{align}
	Combining \eqref{eqsm1} and \eqref{eqsm3} for $n \ge N_0$ with \eqref{eqsm5} and \eqref{eqsm6} for $n<N_0$ and using the
	triangle inequality, $c_1 \ge 1$, and $c_2 \ge 2$, we have by \eqref{eqini}, Lemma~\ref{lem4.2}, and \eqref{eqsm2a},
	where $c_3 := \max\{
	\frac{\mu(1+\theta_\infty \mu^2)}{2}, \frac{c_1 \mu N_0}{2\alpha} \}$ and $c_4=c_4(s)$ denotes the constant from
	Lemma~\ref{lem4.2},
	\begin{align*}
	  |\Psi_1|_s &= \left( \sum_{n=1}^\infty n^{2s} \sup_{t>0} e^{2\alpha t} |\Psi_{1,n}(t)|^2 \right)^{\frac{1}{2}} \nonumber \\
		&\le \sum\limits_{j=1}^3 c_1 \left( \sum_{n=1}^\infty n^{2s} |U_{j,n}(0)|^2 \right)^{\frac{1}{2}} \nonumber \\
		&\hspace*{+5mm} + \sum\limits_{j=1}^3 \frac{1}{12} \left( \sum_{n=1}^\infty n^{2s} \sup_{t>0} e^{2\alpha t} |A_{j,n}(t)|^2
		\right)^{\frac{1}{2}} \nonumber \\
		&\hspace*{+5mm} + \sum\limits_{j=1}^3 2c_2c_3R \left( \sum_{n=1}^\infty n^{2s} \sup_{t>0} e^{2\alpha t} |A_{j,n}(t)|^2
		\right)^{\frac{1}{2}} \nonumber \\
		&\hspace*{+5mm} + c_3\left( \sum_{n=1}^\infty n^{2s} \sup_{t>0} e^{2\alpha t} \left|
		\frac{1}{n}\sum\limits_{k=1}^{n-1} |a_{3,n-k}(t)| k |a_{2,k}(t)| \right|^2 \right)^{\frac{1}{2}} \nonumber \\
		&\hspace*{+5mm} + c_3 \left( \sum_{n=1}^\infty n^{2s} \sup_{t>0} e^{2\alpha t} \left|
		\frac{1}{n} \sum\limits_{l=1}^\infty |a_{3,l+n}(t)| l |a_{2,l}(t)| \right|^2 \right)^{\frac{1}{2}}  \nonumber \\
		&\hspace*{+5mm} + c_3 \left( \sum_{n=1}^\infty n^{2s} \sup_{t>0} e^{2\alpha t} \left|
		\frac{1}{n} \sum\limits_{l=1}^\infty |a_{3,l}(t)| (l+n) |a_{2,l+n}(t)| \right|^2 \right)^{\frac{1}{2}} \nonumber \\
		&\le 3c_1 \delta R + \frac{R}{4}+6 c_2 c_3 R^2 + 3c_3 c_4
		\left( \sum\limits_{n=1}^\infty n^{2s} \sup\limits_{t>0} e^{2\alpha t} |a_{2,n}(t)|^2
		\right)^{\frac{1}{2}} \cdot \nonumber \\
		&\hspace*{+5mm} \left( \sum\limits_{n=1}^\infty n^{2s} \sup\limits_{t>0} e^{2\alpha t} |a_{3,n}(t)|^2
		\right)^{\frac{1}{2}} \nonumber \\
		&\le 3c_1 \delta R + \frac{R}{4} + 6 c_2 c_3 R^2 \nonumber \\
		&\hspace*{+5mm} + 3c_3 c_4
		\left[\sum\limits_{j=1}^3 c_2 \left( \sum\limits_{n=1}^\infty n^{2s} \sup\limits_{t>0} e^{2\alpha t} |A_{j,n}(t)|^2
		\right)^{\frac{1}{2}} \right] \cdot \nonumber \\
		&\hspace*{+5mm} \left[\sum\limits_{j=1}^3 c_2 \left( \sum\limits_{n=1}^\infty n^{2s} \sup\limits_{t>0} e^{2\alpha t}
		|A_{j,n}(t)|^2 \right)^{\frac{1}{2}} \right] \nonumber \\
		&\le 3c_1 \delta R + \frac{R}{4} + 6 c_2 c_3R^2 + 27 c_2^2 c_3 c_4R^2 .
	\end{align*}
	Hence, we have
	\begin{equation}\label{eqsm7}
	  |\Psi_1|_s \le R \quad\mbox{if } \delta \le \frac{1}{12c_1} \mbox{ and } R \le \frac{1}{12 c_2 c_3 + 54 c_2^2 c_3 c_4}
		\, .
	\end{equation}
	Next, we consider $\Psi_j$ for $j=2,3$. For $n \ge N_0$, due to \eqref{kontr2}, \eqref{kontr3}, \eqref{alpha}, and
	Lemma~\ref{calka}, we have
	\begin{align}\label{eqsm8}
	  \sup\limits_{t>0} e^{\alpha t} |\Psi_{j,n}(t)|
		&\le \sup\limits_{t>0} e^{\alpha t} \Big( e^{-\frac{\theta_\infty\mu^2}{2} t} |U_{j,n}(0)| \nonumber \\
		&\hspace*{+5mm} +\int_0^t e^{(-\frac{\theta_\infty\mu^2}{2}(t-\sigma)} |F_{j,n}(\sigma)|d\sigma \Big) \nonumber \\
		&\le |U_{j,n}(0)| + \sup\limits_{t>0} e^{\alpha t} \int_0^t e^{-\frac{\theta_\infty\mu^2}{2}(t-\sigma)}|F_{j,n}(\sigma)|
		d\sigma \nonumber \\
		&\le |U_{j,n}(0)| + \frac{1}{\frac{\theta_\infty\mu^2}{2}-\alpha} \, \sup\limits_{t>0} e^{\alpha t} |F_{j,n}(t)| \nonumber \\
		&\le |U_{j,n}(0)| + \frac{1}{\frac{\theta_\infty\mu^2}{2}- \frac{\theta_\infty\mu^2}{4}} \, \sup\limits_{t>0} e^{\alpha t}
		|F_{j,n}(t)| \nonumber \\
		&\le |U_{j,n}(0)| + \frac{4}{\theta_\infty\mu^2} \, \sup\limits_{t>0} e^{\alpha t} |F_{j,n}(t)|, \qquad n \ge N_0.
  \end{align}
	In view of \eqref{eqsm2}, \eqref{F}, and \eqref{n0}, we get
	\begin{align}\label{eqsm9}
	  \frac{4}{\theta_\infty\mu^2} \, |F_{j,n}|
		&\le \frac{48 (1+\theta_\infty^2)(1+ \mu^4)}{\theta_\infty\mu n}(|A_{1,n}| + |A_{2,n}| + |A_{3,n}|) \nonumber \\
		&\hspace*{+5mm} + \frac{2(2+\theta_\infty \mu^2)}{\theta_\infty} \bigg( \frac{1}{n}\sum\limits_{k=1}^{n-1}
	  |a_{3,n-k}| k |a_{2,k}| + \frac{1}{n} \sum\limits_{l=1}^\infty |a_{3,l+n}| l |a_{2,l}|  \nonumber \\
	  & \hspace*{+5mm} + \frac{1}{n} \sum\limits_{l=1}^\infty |a_{3,l}| (l+n) |a_{2,l+n}| \bigg) \nonumber \\
	  & \hspace*{+5mm} + \frac{8(2+\theta_\infty \mu^2)}{\theta_\infty} \, |h-\theta_\infty| (|A_{1,n}| + |A_{2,n}| + |A_{3,n}|) \nonumber \\
		&\le \frac{1}{12}(|A_{1,n}| + |A_{2,n}| + |A_{3,n}|) \nonumber \\
		&\hspace*{+5mm} + \frac{2(2+\theta_\infty \mu^2)}{\theta_\infty} \bigg( \frac{1}{n}\sum\limits_{k=1}^{n-1}
	  |a_{3,n-k}| k |a_{2,k}| + \frac{1}{n} \sum\limits_{l=1}^\infty |a_{3,l+n}| l |a_{2,l}|  \nonumber \\
	  &\hspace*{+5mm} + \frac{1}{n} \sum\limits_{l=1}^\infty |a_{3,l}| (l+n) |a_{2,l+n}| \bigg) \nonumber \\
		&\hspace*{+5mm} + \frac{8(2+\theta_\infty \mu^2)}{\theta_\infty} \, R (|A_{1,n}| + |A_{2,n}| + |A_{3,n}|) .
	\end{align}
	Combining \eqref{eqsm8} and \eqref{eqsm9} for $n \ge N_0$ with \eqref{eqsm5} and \eqref{eqsm6} for $n<N_0$ and using the
	triangle inequality, $c_1 \ge 1$, and $c_2 \ge 2$, we have by \eqref{eqini}, Lemma~\ref{lem4.2}, and \eqref{eqsm2a},
	where $c_5:= \max\{
	\frac{2(2+\theta_\infty \mu^2)}{\theta_\infty}, \frac{c_1 \mu N_0}{2\alpha} \}$ and $c_4=c_4(s)$ again denotes the
	constant from Lemma~\ref{lem4.2},
	\begin{align*}
	  |\Psi_j|_s &= \left( \sum_{n=1}^\infty n^{2s} \sup_{t>0} e^{2\alpha t} |\Psi_{j,n}(t)|^2 \right)^{\frac{1}{2}} \\
		&\le \sum\limits_{j=1}^3 c_1 \left( \sum_{n=1}^\infty n^{2s} |U_{j,n}(0)|^2 \right)^{\frac{1}{2}} \\
		&\hspace*{+5mm} + \sum\limits_{j=1}^3 \frac{1}{12} \left( \sum_{n=1}^\infty n^{2s} \sup_{t>0} e^{2\alpha t} |A_{j,n}(t)|^2
		\right)^{\frac{1}{2}} \\
		&\hspace*{+5mm} + \sum\limits_{j=1}^3 2 c_2 c_5 R \left( \sum_{n=1}^\infty n^{2s} \sup_{t>0} e^{2\alpha t}
		|A_{j,n}(t)|^2 \right)^{\frac{1}{2}} \\
		&\hspace*{+5mm} + c_5\left( \sum_{n=1}^\infty n^{2s} \sup_{t>0} e^{2\alpha t} \left|
		\frac{1}{n}\sum\limits_{k=1}^{n-1} |a_{3,n-k}(t)| k |a_{2,k}(t)| \right|^2 \right)^{\frac{1}{2}} \\
		&\hspace*{+5mm} + c_5\left( \sum_{n=1}^\infty n^{2s} \sup_{t>0} e^{2\alpha t} \left|
		\frac{1}{n} \sum\limits_{l=1}^\infty |a_{3,l+n}(t)| l |a_{2,l}(t)| \right|^2 \right)^{\frac{1}{2}}  \\
		&\hspace*{+5mm} + c_5\left( \sum_{n=1}^\infty n^{2s} \sup_{t>0} e^{2\alpha t} \left|
		\frac{1}{n} \sum\limits_{l=1}^\infty |a_{3,l}(t)| (l+n) |a_{2,l+n}(t)| \right|^2 \right)^{\frac{1}{2}} \\
		&\le 3c_1 \delta R + \frac{R}{4} + 6 c_2 c_5 R^2 + 3 c_5 c_4
		\left( \sum\limits_{n=1}^\infty n^{2s} \sup\limits_{t>0} e^{2\alpha t} |a_{2,n}(t)|^2
		\right)^{\frac{1}{2}} \cdot \\
		&\hspace*{+5mm} \left( \sum\limits_{n=1}^\infty n^{2s} \sup\limits_{t>0} e^{2\alpha t} |a_{3,n}(t)|^2
		\right)^{\frac{1}{2}} \\
		&\le 3c_1 \delta R + \frac{R}{4} + 6 c_2 c_5 R^2 \\
		&\hspace*{+5mm} + 3 c_5 c_4
		\left[\sum\limits_{j=1}^3 c_2 \left( \sum\limits_{n=1}^\infty n^{2s} \sup\limits_{t>0} e^{2\alpha t} |A_{j,n}(t)|^2
		\right)^{\frac{1}{2}} \right] \cdot \\
		&\hspace*{+5mm} \left[\sum\limits_{j=1}^3 c_2 \left( \sum\limits_{n=1}^\infty n^{2s} \sup\limits_{t>0} e^{2\alpha t}
		|A_{j,n}(t)|^2 \right)^{\frac{1}{2}} \right] \\
		&\le 3c_1 \delta R + \frac{R}{4} + 6 c_2 c_5 R^2 + 27 c_2^2 c_5 c_4 R^2.
	\end{align*}
	Hence, we have
	\begin{equation}\label{eqsm10}
	  |\Psi_j|_s \le R \quad\mbox{for } j=2,3, \mbox{ if } \delta \le \frac{1}{12c_1} \mbox{ and } R \le \frac{1}{12 c_2 c_5
		+ 54 c_2^2 c_5 c_4}  \, .
	\end{equation}
	Finally, for $\xi$ we get due to \eqref{kontr4}, \eqref{eqini}, $s > \frac{1}{2}$, the Cauchy-Schwarz inequality in $l^2$,
	and \eqref{eqsm2a}
	\begin{align*}
	  \sup\limits_{t>0} |\xi(t)-\theta_\infty|
		&\le |\hat{\theta}_0(0)-\theta_\infty| + \frac{\mu}{2} \sup\limits_{t>0} \int_0^t \sum\limits_{l=1}^\infty |a_{3,l}(\sigma)| l
		|a_{2,l}(\sigma)| d\sigma \\
		&\le \delta R + \frac{\mu}{2} \sup\limits_{t>0} \int_0^t e^{-2\alpha \sigma} \sum\limits_{l=1}^\infty l^s
		e^{\alpha \sigma} |a_{3,l}(\sigma)| l^s e^{\alpha \sigma}	|a_{2,l}(\sigma)| d\sigma \\
		&\le \delta R + \frac{\mu}{2} \int_0^\infty e^{-2\alpha \sigma} d\sigma \; \cdot \\
		&\hspace*{+5mm} \sum\limits_{l=1}^\infty \Big( l^s\sup\limits_{t>0} e^{\alpha t} |a_{3,l}(t)| \Big) \Big (l^s \sup\limits_{t>0}
		e^{\alpha t}	|a_{2,l}(t)| \Big) \\
		&\le \delta R + \frac{\mu}{4\alpha} \left( \sum\limits_{n=1}^\infty n^{2s} \sup\limits_{t>0} e^{2\alpha t} |a_{3,n}(t)|^2
		\right)^{\frac{1}{2}} \cdot \\
		&\hspace*{+5mm} \left( \sum\limits_{n=1}^\infty n^{2s} \sup\limits_{t>0} e^{2\alpha t} |a_{2,n}(t)|^2
  	\right)^{\frac{1}{2}}  \\
		&\le \delta R + \frac{\mu}{4\alpha}
		\left[\sum\limits_{j=1}^3 c_2 \left( \sum\limits_{n=1}^\infty n^{2s} \sup\limits_{t>0} e^{2\alpha t} |A_{j,n}(t)|^2
		\right)^{\frac{1}{2}} \right] \cdot \\
		&\hspace*{+5mm} \left[\sum\limits_{j=1}^3 c_2 \left( \sum\limits_{n=1}^\infty n^{2s} \sup\limits_{t>0} e^{2\alpha t}
		|A_{j,n}(t)|^2 \right)^{\frac{1}{2}} \right]	\\
		&\le \delta R + \frac{9\mu c_2^2}{4\alpha} \, R^2 .
	\end{align*}
	Hence, we have
	\begin{equation}\label{eqsm11}
	  \sup\limits_{t>0} |\xi(t)-\theta_\infty|  \le R \quad\mbox{if } \delta \le \frac{1}{4} \mbox{ and } R \le
		\frac{\alpha}{3\mu c_2^2} \, .
	\end{equation}
	Combining this with \eqref{eqsm7}, \eqref{eqsm10}, and \eqref{pX}, we conclude that
	\begin{equation}\label{eqsm12}
	  \begin{split}
		  & \|(\Psi_1, \Psi_2, \Psi_3, \xi-\theta_\infty)\|_X \le R \quad\mbox{if } \delta \le \delta_0:= \min \left\{ \frac{1}{12c_1},
			\frac{1}{4}\right\} \mbox{ and } \\
			&R \le \min\left\{ \frac{1}{12 c_2 c_3 + 54 c_2^2 c_3 c_4}, \frac{1}{12 c_2 c_5 + 54 c_2^2 c_5 c_4},
			\frac{\alpha}{3\mu c_2^2} \right\} .
		\end{split}
	\end{equation}
	Hence, $\mathcal{F}$ is a self mapping of $\overline{B((0,0,0,\theta_\infty),R)}$ with this choice of $\delta$ and $R$.
	
	\textbf{Step 2: Contraction} \\
	Let $\mathcal{A}^{(k)} := (A_1^{(k)}, A_2^{(k)}, A_3^{(k)}, h^{(k)}) \in \overline{B((0,0,0,\theta_\infty),R)}$ for $k=1,2$. In the sequel we use
	the upper index $(k)$ in many quantities to indicate their relation to $\mathcal{A}^{(k)}$. In order to prove that
	$\mathcal{F}$ is a contraction, we first consider $\Psi_1$. For $n \ge N_0$, in view of \eqref{n0} and \eqref{alpha}, we have
	$-n^2 + \theta_\infty \mu^2 \le -\frac{n^2}{2}$ and $\frac{n^2}{2}-\alpha \ge n >0$ as shown before \eqref{eqsm1}. Thus, due to
	\eqref{kontr1} and Lemma~\ref{calka}, we have
	\begin{align}\label{eqco1}
	  &\hspace*{-5mm} \sup\limits_{t>0} e^{\alpha t} |(\Psi^{(1)}_{1,n} - \Psi^{(2)}_{1,n})(t)|  \nonumber \\
		&\le \sup\limits_{t>0} e^{\alpha t} \int_0^t e^{(-n^2+\theta_\infty \mu^2)(t-\sigma)}|(F^{(1)}_{1,n} - F^{(2)}_{1,n})(\sigma)|
		d\sigma \nonumber \\
		&\le \sup\limits_{t>0} e^{\alpha t} \int_0^t e^{-\frac{n^2}{2}(t-\sigma)}|(F^{(1)}_{1,n} - F^{(2)}_{1,n})(\sigma)|
		d\sigma \nonumber \\
		&\le \frac{1}{\frac{n^2}{2}-\alpha} \, \sup\limits_{t>0} e^{\alpha t} |(F^{(1)}_{1,n} - F^{(2)}_{1,n})(t)| \nonumber \\
		&\le \frac{1}{n} \, \sup\limits_{t>0} e^{\alpha t} |(F^{(1)}_{1,n} - F^{(2)}_{1,n})(t)|, \qquad n \ge N_0.
  \end{align}
	In view of $(a^{(1)}_{1,n}- a^{(2)}_{1,n}, a^{(1)}_{2,n}- a^{(2)}_{2,n}, a^{(1)}_{3,n}- a^{(2)}_{3,n})^T= B_n^{-1}
	(A^{(1)}_{1,n}- A^{(2)}_{1,n}, A^{(1)}_{2,n}- A^{(2)}_{2,n}, A^{(1)}_{3,n}- A^{(2)}_{3,n})^T$ and $B_n^{-1} = (C_n^{-1})^T$,
	\eqref{n0} implies $|(B_n^{-1})_{k,l}| \le 2$ and
	\begin{equation}\label{eqco2}
	  |a^{(1)}_{j,n}- a^{(2)}_{j,n}| \le 2 \sum\limits_{l=1}^3 |A^{(1)}_{l,n}- A^{(2)}_{l,n}|, \qquad n \ge N_0, \, j=1,2,3 .
	\end{equation}
	Inserting this into \eqref{F} and using \eqref{n0} and \eqref{eqsm2}, we get
	\begin{align}\label{eqco3}
	  &\hspace*{-5mm} \frac{1}{n} \, |F^{(1)}_{1,n} - F^{(2)}_{1,n}| \nonumber \\
		&\le \frac{6 (1+\theta_\infty^2)(1+ \mu^4)}{n}\sum\limits_{l=1}^3 |A^{(1)}_{l,n}- A^{(2)}_{l,n}| \nonumber \\
		&\hspace*{+5mm} + \frac{\mu(1+\theta_\infty \mu^2)}{2} \bigg( \frac{1}{n}\sum\limits_{k=1}^{n-1}
	  |a^{(1)}_{3,n-k}-a^{(2)}_{3,n-k}| k |a^{(1)}_{2,k}| \nonumber \\
		&\hspace*{+5mm} + \frac{1}{n}\sum\limits_{k=1}^{n-1} |a^{(2)}_{3,n-k}| k |a^{(1)}_{2,k}-a^{(2)}_{2,k}|
		+ \frac{1}{n} \sum\limits_{l=1}^\infty |a^{(1)}_{3,l+n}-a^{(2)}_{3,l+n}| l |a^{(1)}_{2,l}|  \nonumber \\
	  & \hspace*{+5mm} + \frac{1}{n} \sum\limits_{l=1}^\infty |a^{(2)}_{3,l+n}| l |a^{(1)}_{2,l}-a^{(2)}_{2,l}|
		+ \frac{1}{n} \sum\limits_{l=1}^\infty |a^{(1)}_{3,l}-a^{(2)}_{3,l}| (l+n) |a^{(1)}_{2,l+n}| \nonumber \\
		&\hspace*{+5mm} + \frac{1}{n} \sum\limits_{l=1}^\infty |a^{(2)}_{3,l}| (l+n) |a^{(1)}_{2,l+n}-a^{(2)}_{2,l+n}|
		\bigg)\nonumber \\
	  & \hspace*{+5mm} + 2\mu (1+\theta_\infty \mu^2) \Big( |h^{(1)}-h^{(2)}| (|A^{(1)}_{1,n}| + |A^{(1)}_{2,n}| + |A^{(1)}_{3,n}|)
		\nonumber \\
	  & \hspace*{+5mm} + |h^{(2)}-\theta_\infty| \sum\limits_{l=1}^3 |A^{(1)}_{l,n}- A^{(2)}_{l,n}| \Big)\nonumber \\
		&\le \frac{1}{12} \sum\limits_{l=1}^3 |A^{(1)}_{l,n}- A^{(2)}_{l,n}| \nonumber \\
		&\hspace*{+5mm} + \frac{\mu(1+\theta_\infty \mu^2)}{2} \bigg( \frac{1}{n}\sum\limits_{k=1}^{n-1}
	  |a^{(1)}_{3,n-k}-a^{(2)}_{3,n-k}| k |a^{(1)}_{2,k}| \nonumber \\
		&\hspace*{+5mm} + \frac{1}{n}\sum\limits_{k=1}^{n-1} |a^{(2)}_{3,n-k}| k |a^{(1)}_{2,k}-a^{(2)}_{2,k}|
		+ \frac{1}{n} \sum\limits_{l=1}^\infty |a^{(1)}_{3,l+n}-a^{(2)}_{3,l+n}| l |a^{(1)}_{2,l}|  \nonumber \\
	  & \hspace*{+5mm} + \frac{1}{n} \sum\limits_{l=1}^\infty |a^{(2)}_{3,l+n}| l |a^{(1)}_{2,l}-a^{(2)}_{2,l}|
		+ \frac{1}{n} \sum\limits_{l=1}^\infty |a^{(1)}_{3,l}-a^{(2)}_{3,l}| (l+n) |a^{(1)}_{2,l+n}| \nonumber \\
		&\hspace*{+5mm} + \frac{1}{n} \sum\limits_{l=1}^\infty |a^{(2)}_{3,l}| (l+n) |a^{(1)}_{2,l+n}-a^{(2)}_{2,l+n}|
		\bigg)\nonumber \\
	  & \hspace*{+5mm} + 2\mu (1+\theta_\infty \mu^2) \Big( |h^{(1)}-h^{(2)}| (|A^{(1)}_{1,n}| + |A^{(1)}_{2,n}| + |A^{(1)}_{3,n}|)
		\nonumber \\
	  & \hspace*{+5mm} + R \sum\limits_{l=1}^3 |A^{(1)}_{l,n}- A^{(2)}_{l,n}| \Big) .
	\end{align}
	Next, let $n < N_0$. Then \eqref{kontr1a}, \eqref{eqsm4}, and Lemma~\ref{calka} imply for $j = 1,2,3$ and $n < N_0$
	\begin{align}\label{eqco4}
	  &\hspace*{-5mm} \sup\limits_{t>0} e^{\alpha t} |(\Psi^{(1)}_{j,n} - \Psi^{(2)}_{j,n})(t)| \nonumber \\
		&\le c_1 \sup\limits_{t>0} e^{\alpha t} \int_0^t e^{-2\alpha(t-\sigma)}|(g^{(1)}_{3,n}-g^{(2)}_{3,n})(\sigma)|d\sigma
		\nonumber \\
		&\le \frac{c_1}{2\alpha -\alpha} \, \sup\limits_{t>0} e^{\alpha t} |(g^{(1)}_{3,n}-g^{(2)}_{3,n})(t)| \nonumber \\
		&= \frac{c_1}{\alpha} \, \sup\limits_{t>0} e^{\alpha t} |(g^{(1)}_{3,n}-g^{(2)}_{3,n})(t)|, \quad n < N_0 .
	\end{align}
	For any $n < N_0$ we have $(a^{(1)}_{1,n}- a^{(2)}_{1,n}, a^{(1)}_{2,n}- a^{(2)}_{2,n}, a^{(1)}_{3,n}- a^{(2)}_{3,n})^T
	= B_n^{-1}(A^{(1)}_{1,n}- A^{(2)}_{1,n}, A^{(1)}_{2,n}- A^{(2)}_{2,n}, A^{(1)}_{3,n}- A^{(2)}_{3,n})^T$. Since we have
	$|(B_n^{-1})_{k,l}| \le c_2$ for all $n < N_0$ and all $k,l = 1,2,3$ (see before \eqref{eqsm2a}) and $c_2 \ge 2$,
	we obtain from \eqref{eqco2}
	\begin{equation}\label{eqco2a}
	  |a^{(1)}_{j,n}- a^{(2)}_{j,n}| \le c_2 \sum\limits_{l=1}^3 |A^{(1)}_{l,n}- A^{(2)}_{l,n}|, \qquad n \in \N, \, j=1,2,3 .
	\end{equation}
	By \eqref{g}, \eqref{eqsm2a}, and \eqref{eqco2a}, we further have
	\begin{align}\label{eqco5}
	  &\hspace*{-5mm} \frac{c_1}{\alpha} \, |g^{(1)}_{3,n}-g^{(2)}_{3,n}| \nonumber \\
		&\le \frac{c_1 \mu N_0}{2\alpha} \bigg( \frac{1}{n}\sum\limits_{k=1}^{n-1}
	  |a^{(1)}_{3,n-k}-a^{(2)}_{3,n-k}| k |a^{(1)}_{2,k}| \nonumber \\
		&\hspace*{+5mm} + \frac{1}{n}\sum\limits_{k=1}^{n-1} |a^{(2)}_{3,n-k}| k |a^{(1)}_{2,k}-a^{(2)}_{2,k}|
		+ \frac{1}{n} \sum\limits_{l=1}^\infty |a^{(1)}_{3,l+n}-a^{(2)}_{3,l+n}| l |a^{(1)}_{2,l}|  \nonumber \\
	  & \hspace*{+5mm} + \frac{1}{n} \sum\limits_{l=1}^\infty |a^{(2)}_{3,l+n}| l |a^{(1)}_{2,l}-a^{(2)}_{2,l}|
		+ \frac{1}{n} \sum\limits_{l=1}^\infty |a^{(1)}_{3,l}-a^{(2)}_{3,l}| (l+n) |a^{(1)}_{2,l+n}| \nonumber \\
		&\hspace*{+5mm} + \frac{1}{n} \sum\limits_{l=1}^\infty |a^{(2)}_{3,l}| (l+n) |a^{(1)}_{2,l+n}-a^{(2)}_{2,l+n}|
		\bigg)\nonumber \\
	  & \hspace*{+5mm} + \frac{c_1 c_2\mu N_0}{\alpha}\Big( |h^{(1)}-h^{(2)}| (|A^{(1)}_{1,n}| + |A^{(1)}_{2,n}| + |A^{(1)}_{3,n}|)
		\nonumber \\
	  & \hspace*{+5mm} + R \sum\limits_{l=1}^3 |A^{(1)}_{l,n}- A^{(2)}_{l,n}| \Big) , \quad n < N_0 .
	\end{align}
	Combining \eqref{eqco1} and \eqref{eqco3} for $n \ge N_0$ with \eqref{eqco4} and \eqref{eqco5} for $n<N_0$ and using the
	triangle inequality, $c_1 \ge 1$, and $c_2 \ge 2$, we have by Lemma~\ref{lem4.2}, \eqref{eqsm2a}, and \eqref{eqco2a}, where again
	$c_3 := \max\{\frac{\mu(1+\theta_\infty \mu^2)}{2}, \frac{c_1 \mu N_0}{2\alpha} \}$ and $c_4=c_4(s)$ denotes the constant from
	Lemma~\ref{lem4.2},
	\begin{align*}
	  &\hspace*{-5mm} |\Psi^{(1)}_1- \Psi^{(2)}_1|_s \nonumber \\
		&= \left( \sum_{n=1}^\infty n^{2s} \sup_{t>0} e^{2\alpha t} |(\Psi^{(1)}_{1,n} - \Psi^{(2)}_{1,n})(t)|^2 \right)^{\frac{1}{2}}
		\nonumber \\
		&\le \sum\limits_{j=1}^3 \frac{1}{12} \left( \sum_{n=1}^\infty n^{2s} \sup_{t>0} e^{2\alpha t} |(A^{(1)}_{j,n}-
		A^{(2)}_{j,n})(t)|^2 \right)^{\frac{1}{2}} \nonumber \\
		&\hspace*{+5mm} + \sum\limits_{j=1}^3  2 c_2 c_3 R \left( \sum_{n=1}^\infty n^{2s} \sup_{t>0} e^{2\alpha t}
		|(A^{(1)}_{j,n}-A^{(2)}_{j,n})(t)|^2 \right)^{\frac{1}{2}} \nonumber \\
		&\hspace*{+5mm} + \sum\limits_{j=1}^3  2 c_2 c_3 \Big(\sup_{t>0}|(h^{(1)}-h^{(2)})(t)| \Big)
		\left( \sum_{n=1}^\infty n^{2s} \sup_{t>0} e^{2\alpha t} |A^{(1)}_{j,n}(t)|^2 \right)^{\frac{1}{2}} \nonumber \\
		&\hspace*{+5mm} + c_3 \left( \sum_{n=1}^\infty n^{2s} \sup_{t>0} e^{2\alpha t} \left|
		\frac{1}{n}\sum\limits_{k=1}^{n-1} |a^{(1)}_{3,n-k}-a^{(2)}_{3,n-k}| k |a^{(1)}_{2,k}| \right|^2 \right)^{\frac{1}{2}}
		\nonumber \\
		&\hspace*{+5mm} + c_3 \left( \sum_{n=1}^\infty n^{2s} \sup_{t>0} e^{2\alpha t} \left|
		\frac{1}{n}\sum\limits_{k=1}^{n-1} |a^{(2)}_{3,n-k}| k |a^{(1)}_{2,k}-a^{(2)}_{2,k}| \right|^2 \right)^{\frac{1}{2}}
		\nonumber \\
		&\hspace*{+5mm} + c_3 \left( \sum_{n=1}^\infty n^{2s} \sup_{t>0} e^{2\alpha t} \left|
		\frac{1}{n} \sum\limits_{l=1}^\infty |a^{(1)}_{3,l+n}-a^{(2)}_{3,l+n}| l |a^{(1)}_{2,l}| \right|^2 \right)^{\frac{1}{2}}
		\nonumber \\
		&\hspace*{+5mm} + c_3 \left( \sum_{n=1}^\infty n^{2s} \sup_{t>0} e^{2\alpha t} \left|
		\frac{1}{n} \sum\limits_{l=1}^\infty |a^{(2)}_{3,l+n}| l |a^{(1)}_{2,l}-a^{(2)}_{2,l}|  \right|^2 \right)^{\frac{1}{2}}
		\nonumber \\
		&\hspace*{+5mm} + c_3 \left( \sum_{n=1}^\infty n^{2s} \sup_{t>0} e^{2\alpha t} \left|
		\frac{1}{n} \sum\limits_{l=1}^\infty |a^{(1)}_{3,l}-a^{(2)}_{3,l}| (l+n) |a^{(1)}_{2,l+n}| \right|^2 \right)^{\frac{1}{2}}
		\nonumber \\
		&\hspace*{+5mm} + c_3 \left( \sum_{n=1}^\infty n^{2s} \sup_{t>0} e^{2\alpha t} \left|
		\frac{1}{n} \sum\limits_{l=1}^\infty |a^{(2)}_{3,l}| (l+n) |a^{(1)}_{2,l+n}-a^{(2)}_{2,l+n}| \right|^2 \right)^{\frac{1}{2}}
		\nonumber \\
		&\le \frac{1}{4} \, \|\mathcal{A}^{(1)} - \mathcal{A}^{2}\|_X + 6 c_2 c_3 R \|\mathcal{A}^{(1)} - \mathcal{A}^{2}\|_X
		+ 6 c_2 c_3 R \|\mathcal{A}^{(1)} - \mathcal{A}^{2}\|_X \nonumber \\
		&\hspace*{+5mm} + 3c_3 c_4
		\left( \sum\limits_{n=1}^\infty n^{2s} \sup\limits_{t>0} e^{2\alpha t} |a^{(1)}_{2,n}(t)|^2
		\right)^{\frac{1}{2}} \cdot \nonumber \\
		&\hspace*{+5mm} \left( \sum\limits_{n=1}^\infty n^{2s} \sup\limits_{t>0} e^{2\alpha t} |(a^{(1)}_{3,n}- a^{(2)}_{3,n})(t)|^2
		\right)^{\frac{1}{2}} \nonumber \\
		&\hspace*{+5mm} + 3c_3 c_4
		\left( \sum\limits_{n=1}^\infty n^{2s} \sup\limits_{t>0} e^{2\alpha t} |(a^{(1)}_{2,n}-a^{(2)}_{2,n})(t)|^2
		\right)^{\frac{1}{2}} \cdot \nonumber \\
		&\hspace*{+5mm} \left( \sum\limits_{n=1}^\infty n^{2s} \sup\limits_{t>0} e^{2\alpha t} |a^{(2)}_{3,n}(t)|^2
		\right)^{\frac{1}{2}} \nonumber \\
		&\le \left( \frac{1}{4} + 12 c_2 c_3 R \right) \|\mathcal{A}^{(1)} - \mathcal{A}^{2}\|_X \nonumber \\
		&\hspace*{+5mm} + 3 c_3 c_4
		\left[\sum\limits_{j=1}^3 c_2 \left( \sum\limits_{n=1}^\infty n^{2s} \sup\limits_{t>0} e^{2\alpha t} |A^{(1)}_{j,n}(t)|^2
		\right)^{\frac{1}{2}} \right] \cdot \nonumber \\
		&\hspace*{+5mm} \left[\sum\limits_{j=1}^3 c_2 \left( \sum\limits_{n=1}^\infty n^{2s} \sup\limits_{t>0} e^{2\alpha t}
		|(A^{(1)}_{j,n}- A^{(2)}_{j,n})(t)|^2 \right)^{\frac{1}{2}} \right] \nonumber \\
		&\hspace*{+5mm} + 3 c_3 c_4
		\left[\sum\limits_{j=1}^3 c_2 \left( \sum\limits_{n=1}^\infty n^{2s} \sup\limits_{t>0} e^{2\alpha t} |A^{(2)}_{j,n}(t)|^2
		\right)^{\frac{1}{2}} \right] \cdot \nonumber \\
		&\hspace*{+5mm} \left[\sum\limits_{j=1}^3 c_2 \left( \sum\limits_{n=1}^\infty n^{2s} \sup\limits_{t>0} e^{2\alpha t}
		|(A^{(1)}_{j,n}- A^{(2)}_{j,n})(t)|^2 \right)^{\frac{1}{2}} \right] \nonumber \\
		&\le \left( \frac{1}{4} + 12 c_2 c_3 R + 54 c_2^2 c_3 c_4 R\right) \|\mathcal{A}^{(1)} - \mathcal{A}^{2}\|_X .
	\end{align*}
	Hence, we have
	\begin{equation}\label{eqco6}
	  |\Psi^{(1)}_1- \Psi^{(2)}_1|_s \le \frac{1}{2} \|\mathcal{A}^{(1)} - \mathcal{A}^{2}\|_X
		\quad\mbox{if } R \le \frac{1}{4(12 c_2 c_3 + 54 c_2^2 c_3 c_4)}  \, .
	\end{equation}
	Next, we consider $\Psi_j$ for $j=2,3$. For $n \ge N_0$, due to \eqref{kontr2}, \eqref{kontr3}, \eqref{alpha}, and
	Lemma~\ref{calka}, we have
	\begin{align}\label{eqco7}
	  &\hspace*{-5mm}
		\sup\limits_{t>0} e^{\alpha t} |(\Psi^{(1)}_{j,n} - \Psi^{(2)}_{j,n})(t)| \nonumber \\
		&\le \sup\limits_{t>0} e^{\alpha t} \int_0^t e^{(-\frac{\theta_\infty\mu^2}{2}(t-\sigma)} |(F^{(1)}_{j,n} - F^{(2)}_{j,n})
		(\sigma)|d\sigma \nonumber \\
		&\le \frac{1}{\frac{\theta_\infty\mu^2}{2}-\alpha} \, \sup\limits_{t>0} e^{\alpha t} |(F^{(1)}_{j,n} - F^{(2)}_{j,n})(t)|
		\nonumber \\
		&\le \frac{1}{\frac{\theta_\infty\mu^2}{2}- \frac{\theta_\infty\mu^2}{4}} \, \sup\limits_{t>0} e^{\alpha t}
		|(F^{(1)}_{j,n} - F^{(2)}_{j,n})(t)| \nonumber \\
		&\le \frac{4}{\theta_\infty\mu^2} \, \sup\limits_{t>0} e^{\alpha t} |(F^{(1)}_{j,n} - F^{(2)}_{j,n})(t)|, \qquad n \ge N_0.
  \end{align}
	In view of \eqref{eqsm2}, \eqref{eqco2}, \eqref{F}, and \eqref{n0}, we get
	\begin{align}\label{eqco8}
	  &\hspace*{-5mm} \frac{4}{\theta_\infty\mu^2} \, |F^{(1)}_{j,n} - F^{(2)}_{j,n}| \nonumber \\
		&\le \frac{48 (1+\theta_\infty^2)(1+ \mu^4)}{\theta_\infty\mu n} \sum\limits_{l=1}^3 |A^{(1)}_{l,n}- A^{(2)}_{l,n}|
		\nonumber \\
		&\hspace*{+5mm} + \frac{2(2+\theta_\infty \mu^2)}{\theta_\infty} \bigg( \frac{1}{n}\sum\limits_{k=1}^{n-1}
	  |a^{(1)}_{3,n-k}-a^{(2)}_{3,n-k}| k |a^{(1)}_{2,k}| \nonumber \\
		&\hspace*{+5mm} + \frac{1}{n}\sum\limits_{k=1}^{n-1} |a^{(2)}_{3,n-k}| k |a^{(1)}_{2,k}-a^{(2)}_{2,k}|
		+ \frac{1}{n} \sum\limits_{l=1}^\infty |a^{(1)}_{3,l+n}-a^{(2)}_{3,l+n}| l |a^{(1)}_{2,l}|  \nonumber \\
	  & \hspace*{+5mm} + \frac{1}{n} \sum\limits_{l=1}^\infty |a^{(2)}_{3,l+n}| l |a^{(1)}_{2,l}-a^{(2)}_{2,l}|
		+ \frac{1}{n} \sum\limits_{l=1}^\infty |a^{(1)}_{3,l}-a^{(2)}_{3,l}| (l+n) |a^{(1)}_{2,l+n}| \nonumber \\
		&\hspace*{+5mm} + \frac{1}{n} \sum\limits_{l=1}^\infty |a^{(2)}_{3,l}| (l+n) |a^{(1)}_{2,l+n}-a^{(2)}_{2,l+n}|
		\bigg)\nonumber \\
	  & \hspace*{+5mm} + \frac{8(2+\theta_\infty \mu^2)}{\theta_\infty}  \Big( |h^{(1)}-h^{(2)}| (|A^{(1)}_{1,n}|
		+ |A^{(1)}_{2,n}| + |A^{(1)}_{3,n}|)
		\nonumber \\
	  & \hspace*{+5mm} + |h^{(2)}-\theta_\infty| \sum\limits_{l=1}^3 |A^{(1)}_{l,n}- A^{(2)}_{l,n}| \Big)\nonumber \\
		&\le \frac{1}{12} \sum\limits_{l=1}^3 |A^{(1)}_{l,n}- A^{(2)}_{l,n}| \nonumber \\
		&\hspace*{+5mm} + \frac{2(2+\theta_\infty \mu^2)}{\theta_\infty} \bigg( \frac{1}{n}\sum\limits_{k=1}^{n-1}
	  |a^{(1)}_{3,n-k}-a^{(2)}_{3,n-k}| k |a^{(1)}_{2,k}| \nonumber \\
		&\hspace*{+5mm} + \frac{1}{n}\sum\limits_{k=1}^{n-1} |a^{(2)}_{3,n-k}| k |a^{(1)}_{2,k}-a^{(2)}_{2,k}|
		+ \frac{1}{n} \sum\limits_{l=1}^\infty |a^{(1)}_{3,l+n}-a^{(2)}_{3,l+n}| l |a^{(1)}_{2,l}|  \nonumber \\
	  & \hspace*{+5mm} + \frac{1}{n} \sum\limits_{l=1}^\infty |a^{(2)}_{3,l+n}| l |a^{(1)}_{2,l}-a^{(2)}_{2,l}|
		+ \frac{1}{n} \sum\limits_{l=1}^\infty |a^{(1)}_{3,l}-a^{(2)}_{3,l}| (l+n) |a^{(1)}_{2,l+n}| \nonumber \\
		&\hspace*{+5mm} + \frac{1}{n} \sum\limits_{l=1}^\infty |a^{(2)}_{3,l}| (l+n) |a^{(1)}_{2,l+n}-a^{(2)}_{2,l+n}|
		\bigg)\nonumber \\
	  & \hspace*{+5mm} + \frac{8(2+\theta_\infty \mu^2)}{\theta_\infty} \Big( |h^{(1)}-h^{(2)}| (|A^{(1)}_{1,n}|
		+ |A^{(1)}_{2,n}| + |A^{(1)}_{3,n}|)
		\nonumber \\
	  & \hspace*{+5mm} + R \sum\limits_{l=1}^3 |A^{(1)}_{l,n}- A^{(2)}_{l,n}| \Big) .
	\end{align}
	Combining \eqref{eqco7} and \eqref{eqco8} for $n \ge N_0$ with \eqref{eqco4} and \eqref{eqco5} for $n<N_0$ and using the
	triangle inequality, $c_1 \ge 1$, and $c_2 \ge 2$, we have by Lemma~\ref{lem4.2}, \eqref{eqsm2a}, and \eqref{eqco2a}, where again $c_5 := \max\{
	\frac{2(2+\theta_\infty \mu^2)}{\theta_\infty}, \frac{c_1 \mu N_0}{2\alpha} \}$ and $c_4=c_4(s)$ denotes the constant from
	Lemma~\ref{lem4.2},
	\begin{align*}
	  &\hspace*{-5mm} |\Psi^{(1)}_j- \Psi^{(2)}_j|_s \nonumber \\
		&= \left( \sum_{n=1}^\infty n^{2s} \sup_{t>0} e^{2\alpha t} |(\Psi^{(1)}_{j,n} - \Psi^{(2)}_{j,n})(t)|^2 \right)^{\frac{1}{2}}
		\nonumber \\
		&\le \sum\limits_{j=1}^3 \frac{1}{12} \left( \sum_{n=1}^\infty n^{2s} \sup_{t>0} e^{2\alpha t} |(A^{(1)}_{j,n}-
		A^{(2)}_{j,n})(t)|^2 \right)^{\frac{1}{2}} \nonumber \\
		&\hspace*{+5mm} + \sum\limits_{j=1}^3  2 c_2 c_5 R \left( \sum_{n=1}^\infty n^{2s} \sup_{t>0} e^{2\alpha t}
		|(A^{(1)}_{j,n}-A^{(2)}_{j,n})(t)|^2 \right)^{\frac{1}{2}} \nonumber \\
		&\hspace*{+5mm} + \sum\limits_{j=1}^3  2 c_2 c_5 \Big(\sup_{t>0}|(h^{(1)}-h^{(2)})(t)| \Big)
		\left( \sum_{n=1}^\infty n^{2s} \sup_{t>0} e^{2\alpha t} |A^{(1)}_{j,n}(t)|^2 \right)^{\frac{1}{2}} \nonumber \\
		&\hspace*{+5mm} + c_5 \left( \sum_{n=1}^\infty n^{2s} \sup_{t>0} e^{2\alpha t} \left|
		\frac{1}{n}\sum\limits_{k=1}^{n-1} |a^{(1)}_{3,n-k}-a^{(2)}_{3,n-k}| k |a^{(1)}_{2,k}| \right|^2 \right)^{\frac{1}{2}}
		\nonumber \\
		&\hspace*{+5mm} + c_5 \left( \sum_{n=1}^\infty n^{2s} \sup_{t>0} e^{2\alpha t} \left|
		\frac{1}{n}\sum\limits_{k=1}^{n-1} |a^{(2)}_{3,n-k}| k |a^{(1)}_{2,k}-a^{(2)}_{2,k}| \right|^2 \right)^{\frac{1}{2}}
		\nonumber \\
		&\hspace*{+5mm} + c_5 \left( \sum_{n=1}^\infty n^{2s} \sup_{t>0} e^{2\alpha t} \left|
		\frac{1}{n} \sum\limits_{l=1}^\infty |a^{(1)}_{3,l+n}-a^{(2)}_{3,l+n}| l |a^{(1)}_{2,l}| \right|^2 \right)^{\frac{1}{2}}
		\nonumber \\
		&\hspace*{+5mm} + c_5 \left( \sum_{n=1}^\infty n^{2s} \sup_{t>0} e^{2\alpha t} \left|
		\frac{1}{n} \sum\limits_{l=1}^\infty |a^{(2)}_{3,l+n}| l |a^{(1)}_{2,l}-a^{(2)}_{2,l}|  \right|^2 \right)^{\frac{1}{2}}
		\nonumber \\
		&\hspace*{+5mm} + c_5 \left( \sum_{n=1}^\infty n^{2s} \sup_{t>0} e^{2\alpha t} \left|
		\frac{1}{n} \sum\limits_{l=1}^\infty |a^{(1)}_{3,l}-a^{(2)}_{3,l}| (l+n) |a^{(1)}_{2,l+n}| \right|^2 \right)^{\frac{1}{2}}
		\nonumber \\
		&\hspace*{+5mm} + c_5 \left( \sum_{n=1}^\infty n^{2s} \sup_{t>0} e^{2\alpha t} \left|
		\frac{1}{n} \sum\limits_{l=1}^\infty |a^{(2)}_{3,l}| (l+n) |a^{(1)}_{2,l+n}-a^{(2)}_{2,l+n}| \right|^2 \right)^{\frac{1}{2}}
		\nonumber \\
		&\le \frac{1}{4} \, \|\mathcal{A}^{(1)} - \mathcal{A}^{2}\|_X + 6 c_2 c_5 R \|\mathcal{A}^{(1)} - \mathcal{A}^{2}\|_X
		+ 6 c_2 c_5 R \|\mathcal{A}^{(1)} - \mathcal{A}^{2}\|_X \nonumber \\
		&\hspace*{+5mm} + 3c_5 c_4
		\left( \sum\limits_{n=1}^\infty n^{2s} \sup\limits_{t>0} e^{2\alpha t} |a^{(1)}_{2,n}(t)|^2
		\right)^{\frac{1}{2}} \cdot \nonumber \\
		&\hspace*{+5mm} \left( \sum\limits_{n=1}^\infty n^{2s} \sup\limits_{t>0} e^{2\alpha t} |(a^{(1)}_{3,n}- a^{(2)}_{3,n})(t)|^2
		\right)^{\frac{1}{2}} \nonumber \\
		&\hspace*{+5mm} + 3c_5 c_4
		\left( \sum\limits_{n=1}^\infty n^{2s} \sup\limits_{t>0} e^{2\alpha t} |(a^{(1)}_{2,n}-a^{(2)}_{2,n})(t)|^2
		\right)^{\frac{1}{2}} \cdot \nonumber \\
		&\hspace*{+5mm} \left( \sum\limits_{n=1}^\infty n^{2s} \sup\limits_{t>0} e^{2\alpha t} |a^{(2)}_{3,n}(t)|^2
		\right)^{\frac{1}{2}} \nonumber \\
		&\le \left( \frac{1}{4} + 12 c_2 c_5 R \right) \|\mathcal{A}^{(1)} - \mathcal{A}^{2}\|_X \nonumber \\
		&\hspace*{+5mm} + 3c_5 c_4
		\left[\sum\limits_{j=1}^3 c_2 \left( \sum\limits_{n=1}^\infty n^{2s} \sup\limits_{t>0} e^{2\alpha t} |A^{(1)}_{j,n}(t)|^2
		\right)^{\frac{1}{2}} \right] \cdot \nonumber \\
		&\hspace*{+5mm} \left[\sum\limits_{j=1}^3 c_2 \left( \sum\limits_{n=1}^\infty n^{2s} \sup\limits_{t>0} e^{2\alpha t}
		|(A^{(1)}_{j,n}- A^{(2)}_{j,n})(t)|^2 \right)^{\frac{1}{2}} \right] \nonumber \\
		&\hspace*{+5mm} + 3c_5 c_4
		\left[\sum\limits_{j=1}^3 c_2 \left( \sum\limits_{n=1}^\infty n^{2s} \sup\limits_{t>0} e^{2\alpha t} |A^{(2)}_{j,n}(t)|^2
		\right)^{\frac{1}{2}} \right] \cdot \nonumber \\
		&\hspace*{+5mm} \left[\sum\limits_{j=1}^3 c_2 \left( \sum\limits_{n=1}^\infty n^{2s} \sup\limits_{t>0} e^{2\alpha t}
		|(A^{(1)}_{j,n}- A^{(2)}_{j,n})(t)|^2 \right)^{\frac{1}{2}} \right] \nonumber \\
		&\le \left( \frac{1}{4} + 12 c_2 c_5 R + 54 c_2^2 c_5 c_4 R\right) \|\mathcal{A}^{(1)} - \mathcal{A}^{2}\|_X .
	\end{align*}
	Hence, we have
	\begin{equation}\label{eqco9}
	  \begin{split}
	  & |\Psi^{(1)}_j- \Psi^{(2)}_j|_s \le \frac{1}{2} \|\mathcal{A}^{(1)} - \mathcal{A}^{2}\|_X
		\quad\quad\mbox{for } j=2,3, \\
		&\mbox{ if } R \le \frac{1}{4(12 c_2 c_5 + 54 c_2^2 c_5 c_4)}  \, .
		\end{split}
	\end{equation}
	Finally, for $\xi$ we get due to \eqref{kontr4}, $s > \frac{1}{2}$, the Cauchy-Schwarz inequality in $l^2$,
	\eqref{eqsm2a}, and \eqref{eqco2a} 
	\begin{align*}
	  &\hspace*{-5mm} \sup\limits_{t>0} |(\xi^{(1)}- \xi^{(2)})(t)| \\
		&\le \frac{\mu}{2} \sup\limits_{t>0} \int_0^t \sum\limits_{l=1}^\infty \Big( |(a^{(1)}_{3,l}-a^{(2)}_{3,l})(\sigma)| l
		|a^{(1)}_{2,l}(\sigma)| \\
		&\hspace*{+5mm} + |a^{(2)}_{3,l}(\sigma)| l |(a^{(1)}_{2,l}-a^{(2)}_{2,l})(\sigma)| \Big) d\sigma \\
		&\le \frac{\mu}{2} \sup\limits_{t>0} \int_0^t e^{-2\alpha \sigma} \sum\limits_{l=1}^\infty \Big( l^s
		e^{\alpha \sigma} |(a^{(1)}_{3,l}-a^{(2)}_{3,l})(\sigma)| l^s e^{\alpha \sigma}	|a^{(1)}_{2,l}(\sigma)| \\
		&\hspace*{+5mm} + l^s
		e^{\alpha \sigma} |a^{(2)}_{3,l}(\sigma)| l^s e^{\alpha \sigma}	|(a^{(1)}_{2,l}-a^{(2)}_{2,l})(\sigma)| \Big)d\sigma \\
		&\le \frac{\mu}{2} \int_0^\infty e^{-2\alpha \sigma} d\sigma \; \cdot \\
		&\hspace*{+5mm} \bigg[ \sum\limits_{l=1}^\infty \Big( l^s\sup\limits_{t>0} e^{\alpha t} |(a^{(1)}_{3,l}-a^{(2)}_{3,l})(t)|
		\Big) \Big (l^s \sup\limits_{t>0} e^{\alpha t}	|a^{(1)}_{2,l}(t)| \Big) \\
		&\hspace*{+5mm} +\sum\limits_{l=1}^\infty \Big( l^s\sup\limits_{t>0} e^{\alpha t} |a^{(2)}_{3,l}(t)|
		\Big) \Big (l^s \sup\limits_{t>0} e^{\alpha t}	|(a^{(1)}_{2,l}-a^{(2)}_{2,l})(t)| \Big) \bigg]\\
		&\le \frac{\mu}{4\alpha} \left( \sum\limits_{n=1}^\infty n^{2s} \sup\limits_{t>0} e^{2\alpha t} |a^{(1)}_{2,n}(t)|^2
		\right)^{\frac{1}{2}} \cdot  \\
		&\hspace*{+5mm} \left( \sum\limits_{n=1}^\infty n^{2s} \sup\limits_{t>0} e^{2\alpha t} |(a^{(1)}_{3,n}- a^{(2)}_{3,n})(t)|^2
		\right)^{\frac{1}{2}} \\
		&\hspace*{+5mm} +  \frac{\mu}{4\alpha}
		\left( \sum\limits_{n=1}^\infty n^{2s} \sup\limits_{t>0} e^{2\alpha t} |(a^{(1)}_{2,n}-a^{(2)}_{2,n})(t)|^2
		\right)^{\frac{1}{2}} \cdot \nonumber \\
		&\hspace*{+5mm} \left( \sum\limits_{n=1}^\infty n^{2s} \sup\limits_{t>0} e^{2\alpha t} |a^{(2)}_{3,n}(t)|^2
		\right)^{\frac{1}{2}} \\
		&\le \frac{\mu}{4\alpha} \left[\sum\limits_{j=1}^3 c_2 \left( \sum\limits_{n=1}^\infty n^{2s} \sup\limits_{t>0} e^{2\alpha t}
		|A^{(1)}_{j,n}(t)|^2 \right)^{\frac{1}{2}} \right] \cdot \\
		&\hspace*{+5mm} \left[\sum\limits_{j=1}^3 c_2 \left( \sum\limits_{n=1}^\infty n^{2s} \sup\limits_{t>0} e^{2\alpha t}
		|(A^{(1)}_{j,n}- A^{(2)}_{j,n})(t)|^2 \right)^{\frac{1}{2}} \right] \\
		&\hspace*{+5mm} +\frac{\mu}{4\alpha}
		\left[\sum\limits_{j=1}^3 c_2\left( \sum\limits_{n=1}^\infty n^{2s} \sup\limits_{t>0} e^{2\alpha t} |A^{(2)}_{j,n}(t)|^2
		\right)^{\frac{1}{2}} \right] \cdot  \\
		&\hspace*{+5mm} \left[\sum\limits_{j=1}^3 c_2 \left( \sum\limits_{n=1}^\infty n^{2s} \sup\limits_{t>0} e^{2\alpha t}
		|(A^{(1)}_{j,n}- A^{(2)}_{j,n})(t)|^2 \right)^{\frac{1}{2}} \right] \\
		&\le \frac{9\mu c_2^2}{2\alpha} \, R \, \|\mathcal{A}^{(1)} - \mathcal{A}^{2}\|_X .
	\end{align*}
	Hence, we have
	\begin{equation}\label{eqco10}
	  \sup\limits_{t>0} |(\xi^{(1)}- \xi^{(2)})(t)| \le \frac{1}{2} \|\mathcal{A}^{(1)} - \mathcal{A}^{2}\|_X
		\quad\mbox{if } R \le \frac{\alpha}{9\mu c_2^2} \, .
	\end{equation}
	Combining this with \eqref{eqco6}, \eqref{eqco9}, and \eqref{pX}, we conclude that
	\begin{equation}\label{eqco11}
	  \begin{split}
		  & \|(\Psi^{(1)}_1, \Psi^{(1)}_2, \Psi^{(1)}_3, \xi^{(1)})- (\Psi^{(2)}_1, \Psi^{(2)}_2, \Psi^{(2)}_3, \xi^{(2)})\|_X \le
			\frac{1}{2} \|\mathcal{A}^{(1)} - \mathcal{A}^{2}\|_X \\
			&\mbox{if } R \le R_0:= \min\left\{\frac{1}{4(12 c_2 c_3 + 54 c_2^2 c_3 c_4)}, \frac{1}{4(12 c_2 c_5
			+ 54 c_2^2 c_5 c_4)}, \frac{\alpha}{9\mu c_2^2}\right\} .
		\end{split}
	\end{equation}
	Hence, $\mathcal{F}$ is a contraction on $\overline{B((0,0,0,\theta_\infty),R)}$ with this choice of $R$. Combining this with \eqref{eqsm12} and
	choosing $\delta \in (0, \delta_0]$ and $R \in (0, R_0]$, $\mathcal{F}$ has a unique fixed
	point in $\overline{B((0,0,0,\theta_\infty),R)}$ by Banach's fixed point theorem.
\end{proof}

\section{Conclusions}\label{conlusion}
This section is devoted to the proof of Theorem \ref{main}. In the previous section we constructed a solution fixed point in a closed ball $\overline{B((0,0,0,\theta_\infty), R)}\subset X$ for $R$ small enough, where $X$ consists of vectors $(U_1, U_2, U_3, \hat{\theta}_0)$, where $U_j, j=1,2,3$, are sequences of complex valued functions defined for $t\in [0,\infty)$, $\hat{\theta}_0$ is a complex valued function of time $t\in [0,\infty)$.

We notice first that by Proposition \ref{prop5.2}, for initial data $(U_1(0), U_2(0), U_3(0),$ $\hat{\theta}_0(0))$ in a ball $\overline{B((0,0,0,\theta_\infty), \delta R)}\subset X$ there exists a unique vector function $(U_1, U_2, U_3, \hat{\theta}_0)$ in a ball $\overline{B((0,0,0,\theta_\infty), R)}\subset X$ being a fixed point of $\mathcal{F}$ defined in \eqref{kontrakcja}--\eqref{g}. We define $(n\hat{u}_n, \hat{v}_n, \hat{\theta}_n)^T := B_n^{-1} (U_{1,n},U_{2,n},U_{3,n})^T$ for $n \in \N$, where for $n \ge N_0$, with $N_0$ defined in \eqref{n0}, we have $B_n = C_n^T$ with $C_n$ from \eqref{cd}, while, for
$n < N_0$, $B_n$ is defined before \eqref{kontr1a}. For $n \ge N_0$ and $j=1,2,3,$ $U_{j,n}$ satisfies
\eqref{sysproj} with $F_{j,n}$ given by \eqref{sysprojf}. Then the definition of $B_n$ implies that \eqref{proj} is satisfied and
thus $(n\hat{u}_n, \hat{v}_n, \hat{\theta}_n)$ solves \eqref{nonlinsys} for $n \ge N_0$. For $n < N_0$ the definition of $B_n$
and \eqref{kontr1a} imply that $(n\hat{u}_n, \hat{v}_n, \hat{\theta}_n)$ solves \eqref{nonlinsys} as well. Hence,
$(n\hat{u}_n, \hat{v}_n, \hat{\theta}_n)$ is a solution to \eqref{nonlinsys} for $n \in \N$. In view of \eqref{kontr4} this means
that $\hat{u}_n$, $\hat{\theta}_n$, $n \in \N$, and $\hat{\theta}_0$ solve \eqref{f2} and that $\hat{u}_n^\prime = \hat{v}_n$
and that therefore indeed all these functions are real valued.

Next, we recall that by \eqref{eqsm2a}, in view of $(n\hat{u}_n, \hat{v}_n, \hat{\theta}_n)^T= B_n^{-1} (U_{1,n},U_{2,n},$ $U_{3,n})^T$, we obtain
\begin{equation}\label{eqsm21}
\max\{|n\hat{u}_n|, |\hat{v}_n|, |\hat{\theta}_n|\}\leq c_2 (|U_{1,n}|+ |U_{2,n}| + |U_{3,n}|), \qquad n \in \N.
\end{equation} 
We observe that for $j=1,2,3$ with $\alpha$ from \eqref{alpha} and $s \in (\frac{3}{4},1)$
\[
\sup_{t>0} e^{\alpha t}\sqrt{\sum_{n=1}^\infty n^{2s} |U_{j,n}(t)|^2}\leq |U_{j}|_s\leq R,
\]
hence
\begin{equation}\label{7.1}
\sqrt{\sum_{n=1}^\infty n^{2s} |U_{j,n}(t)|^2}\leq  R e^{-\alpha t} \quad\mbox{for all } t \ge 0.
\end{equation}
Next, we notice that by \eqref{f2}, \eqref{theta_in}, and \eqref{zbwszthe} we have $\hat{\theta}_0(t) \to \theta_\infty$ as
$t \to \infty$ and hence \eqref{f2}, $s > \frac{1}{2}$, the Cauchy-Schwarz inequality in $l^2$, \eqref{eqsm21}, and \eqref{7.1}
imply
\begin{align}\label{thet}
|\hat{\theta_0}(t)-\theta_\infty|
&=\left|\frac{\mu}{2}\int_t^\infty \sum_{l=1}^\infty \hat{\theta}_l(\sigma)l\hat{v}_l(\sigma) d\sigma\right|
\nonumber \\
&\le \frac{\mu}{2}\int_t^\infty \sum_{l=1}^\infty l^s |\hat{\theta}_l(\sigma)| l^s|\hat{v}_l(\sigma)| d\sigma
\nonumber \\
&\le \frac{\mu}{2}\int_t^\infty \left(\sum_{l=1}^\infty l^{2s} |\hat{\theta}_l(\sigma)|^2 \right)^{\frac{1}{2}}
\left(\sum_{l=1}^\infty l^{2s} |\hat{v}_l(\sigma)|^2 \right)^{\frac{1}{2}} d\sigma \nonumber \\
&\le \frac{\mu}{2}\int_t^\infty  \left[ \sum\limits_{j=1}^3 c_2 \left(\sum_{l=1}^\infty l^{2s} |U_{j,l}(\sigma)|^2
\right)^{\frac{1}{2}} \right]^2 d\sigma \nonumber \\
&\le \frac{9\mu c_2^2 R^2}{2} \int_t^\infty e^{-2\alpha \sigma} d \sigma = \frac{9\mu c_2^2 R^2}{4\alpha} e^{-2\alpha t},
\quad t \ge 0.
\end{align} 
Consequently, we have the exponential decay of $\hat{\theta}_0$ towards $\theta_\infty$. But, since $\int_0^\pi \theta(t,x) dx=\pi \hat{\theta}_0(t)$, we conclude that for $s \in (\frac{3}{4},1)$
\[
\|\theta(t,\cdot)-\theta_\infty\|_{H^s(0,\pi)}\leq Ce^{-\alpha t}.
\]
Indeed, by \eqref{eqsm21} and \eqref{7.1} we have
\begin{align}\label{thety}
\left(\sum_{l=1}^\infty l^{2s}|\hat{\theta}_l (t)|^2 \right)^{\frac{1}{2}}
&\leq c_2 \left(\sum_{l=1}^\infty l^{2s}(|U_{1,l}|+|U_{2,l}|+|U_{3,l}|)^2 \right)^{\frac{1}{2}} \nonumber \\
&\leq 3c_2 R e^{-\alpha t}, \quad t \ge 0.
\end{align} 
Consequently, since on $(0,\pi)$ the Fourier expansion is $\theta(t,\cdot)-\theta_\infty=\hat{\theta}_0(t)-\theta_\infty +\sum_{l=1}^\infty \hat{\theta}_l(t) \cos (lx)$, due to \eqref{thet} and \eqref{thety}, we have
\begin{equation}\label{temperatura}
\|\theta(t,\cdot)-\theta_\infty\|_{H^s(0,\pi)}^2= |\hat{\theta}_0(t)-\theta_\infty|^2+\sum_{l=1}^\infty l^{2s} |\hat{\theta}_l(t)|^2\leq c\, e^{-2\alpha t}.
\end{equation}
Similarly, by \eqref{eqsm21} and \eqref{7.1} we have
\begin{align}\label{position}
&\hspace*{-5mm} \|u_x(t,\cdot)\|_{H^s(0,\pi)}^2+\|u_t(t,\cdot)\|_{H^s(0,\pi)}^2 \nonumber \\
&= \sum_{l=1}^\infty l^{2s} |l\hat{u}_l(t)|^2 + \sum_{l=1}^\infty l^{2s} |\hat{v}_l(t)|^2 \nonumber \\
&\leq c \sum_{l=1}^\infty l^{2s}(|U_{1,l}|+|U_{2,l}|+|U_{3,l}|)^2 \leq c\, e^{-2\alpha t}.
\end{align}
To summarize, we see that for initial data $(u_0,v_0,\theta_0)$ such that $((u_0)_x,v_0,\theta_0)$ are in the ball considered in Proposition \ref{prop5.2}, the solution emanating from it decays exponentially towards $(0,0,\theta_\infty)$. But, by Proposition \ref{zbieznosc}, we see that for any data satisfying \eqref{regularity} (and positivity of the initial temperature), there exists a time $t_0$ such that $(u_x(t_0,\cdot), u_t(t_0,\cdot), \theta(t_0,\cdot))$ is contained in an arbitrary small ball around $(0,0,\theta_\infty)$, hence when passing to the new variables $(U_1,U_2,U_3,\hat{\theta}_0)$, we see that, at the intermittent time $t_0$, our solution is in $\overline{B((0,0,0,\theta_\infty), \delta R)}$. Next, \eqref{temperatura}, \eqref{position}, the Poincar\'e inequality $\|u(t,\cdot)\|_{L^2(0,\pi)}\leq c\|u_{x}(t,\cdot)\|_{L^2(0,\pi)}$ and simple interpolation are applicable to our solution and the exponential convergence in Theorem \ref{main} holds, the proof is finished.

At the end, let us comment on additional time decay we obtained. 
We notice the solution, obtained via a fixed point procedure, is of the form
\begin{align*}
U_{1,n}(t)&:=e^{(-n^2+\theta_\infty\mu^2) t} U_{1,n}(0)+\int_0^te^{(-n^2+\theta_\infty \mu^2)(t-\sigma)}F_{1,n}(\sigma)d\sigma,
\\
U_{2,n}(t)&:=e^{(-ni-\frac{\theta_\infty\mu^2}{2}) t} U_{2,n}(0)+\int_0^t e^{(-ni-\frac{\theta_\infty\mu^2}{2})(t-\sigma)}
F_{2,n}(\sigma)d\sigma, \\
U_{3,n}(t)&:=e^{(ni-\frac{\theta_\infty\mu^2}{2}) t} U_{3,n}(0)+\int_0^t e^{(ni-\frac{\theta_\infty\mu^2}{2})(t-\sigma)}
F_{3,n}(\sigma)d\sigma
\end{align*}
for $ n \ge N_0$, while for $n < N_0$ we have
$$\begin{pmatrix} U_{1,n}(t) \\ U_{2,n}(t) \\ U_{3,n}(t) \end{pmatrix}
:= B_n \left( e^{t A_{n,\theta_\infty}} B_n^{-1}\begin{pmatrix} U_{1,n}(0) \\ U_{2,n}(0) \\ U_{3,n}(0) \end{pmatrix}
+\int_0^t e^{(t-\sigma)A_{n,\theta_\infty}} g_n(\sigma) d\sigma \right)$$ 
Moreover, we also obtain that $\hat{\theta}_0(t)$ converges exponentially towards $\theta_\infty$, see \eqref{thet}. While in the case of $U_1$, an exponential time decay is not surprising, the same happens to $U_2$ and $U_3$, the additional factor $e^{-\frac{\theta_\infty\mu^2}{2}t}$ or $e^{-2\alpha t}$ appearing at linear parts of Duhamel's formula is not initially expected, since we deal with a mixed type problem involving the wave equation.

\subsection*{Acknowledgment}
J.J. was supported by ERC project INSOLIT. He thanks the Institute of Mathematics, Polish Academy of Sciences, for its hospitality.

\subsection*{Conflict of interest statement}
The authors have no conflicts of interest to declare that are relevant to the content of this article.

\subsection*{Data availability statement}
Data sharing not applicable to this article as no datasets were generated or analysed during the current study.



\end{document}